\documentclass[11pt, a4paper, twoside,leqno]{amsart}
\usepackage[centering, totalwidth = 380pt, totalheight = 600pt]{geometry}
\usepackage{amssymb, amsmath, amsthm, listings, xspace}
\usepackage{enumerate, microtype, stmaryrd, url} 
\usepackage[latin1]{inputenc}
\usepackage[arrow, matrix, tips, curve, graph, rotate]{xy}
\usepackage[retainorgcmds]{IEEEtrantools}
 \usepackage{color}
\definecolor{daRemarkgreen}{rgb}{0,0.45,0}
\usepackage[colorlinks,citecolor=daRemarkgreen,linkcolor=daRemarkgreen]{hyperref}
\usepackage{lscape}
\SelectTips{cm}{10}
\newcommand{\cd}[2][]{\vcenter{\hbox{\xymatrix#1{#2}}}}

\newcommand{\cat}[1]{\mathbf{#1}}

\makeatletter
\def\matrixobject@{%
   \edef \next@{={\DirectionfromtheDirection@ }}%
   \expandafter \toks@ \next@ \plainxy@
   \let\xy@@ix@=\xyq@@toksix@
   \xyFN@ \OBJECT@}
\let\xy@entry@@norm=\entry@@norm
\def\entry@@norm@patched{%
   \let\object@=\matrixobject@
   \xy@entry@@norm }
\AtBeginDocument{\let\entry@@norm\entry@@norm@patched}
\makeatother

\newdir{(}{{}*!<0em,-.14em>-\cir<.14em>{l^r}}
\newdir{ (}{{}*!/-5pt/\dir{(}}
\newdir{ >}{{}*!/-5pt/\dir{>}}

\theoremstyle{definition}

\numberwithin{equation}{section}
\theoremstyle{plain}
\newtheorem{Theorem}{Theorem}[section]

\newtheorem{ConjectureF}[Theorem]{Faithfulness Conjecture}
\newtheorem{Corollary}[Theorem]{Corollary}
\newtheorem{Proposition}[Theorem]{Proposition}
\newtheorem{Lemma}[Theorem]{Lemma}

\theoremstyle{definition}

\newtheorem{Definition}[Theorem]{Definition}

\newtheorem{Examples}[Theorem]{Examples}
\newtheorem{Example}[Theorem]{Example}
\newtheorem{Remark}[Theorem]{Remark}

\newtheorem{Notation}[Theorem]{Notation}

\renewcommand{\phi}{\varphi}
\newcommand{\A}{{\mathcal A}}
\newcommand{\B}{{\mathcal B}}
\newcommand{\C}{{\mathcal C}}
\newcommand{\D}{{\mathcal D}}
\newcommand{\E}{{\mathcal E}}

\newcommand{\G}{{\mathcal G}}

\newcommand{\I}{{\mathcal I}}

\newcommand{\V}{{\mathcal V}}

\newcommand{\atwo}{{\mathbf 2}}
\newcommand{\f}[1]{\mathbb #1}
\newcommand{\Cat}{{\cat{Cat}}}

\newcommand{\CAT}{{\cat{CAT}}}
\newcommand{\Set}{{\cat{Set}}}

\newcommand{\Mod}{\textnormal{Mod}}
\newcommand{\Alg}{\textnormal{Alg}}

\newcommand{\GTh}{{\Theta_{0}\textnormal{-Th}}}

\newcommand{\Ainj}{\ensuremath{{\bbi}\textnormal{nj}\xspace}}

\newcommand{\cc}{\ensuremath{\mathcal C}\xspace}

\newcommand{\bbi}{\ensuremath{\mathbb I}\xspace}

\newcommand{\sig}{\cat{Sig}_\A(\E)}

\renewcommand{\th}{\cat{Th}_\A(\E)}
\newcommand{\nerv}{\cat{Mnd}_\A(\E)}
\newcommand{\mnd}{\cat{Mnd}(\E)}

\begin{document}
\leftmargini=2em \title{Iterated algebraic injectivity and the faithfulness conjecture}
\author{John Bourke}
\address{Department of Mathematics and Statistics, Masaryk University, Kotl\'a\v rsk\'a 2, Brno 61137, Czech Republic}
\email{bourkej@math.muni.cz}
\subjclass[2000]{Primary: 18D05, 18C10, 55U35}
\date{\today}
\maketitle

\maketitle
\begin{abstract}
Algebraic injectivity was introduced to capture homotopical structures like algebraic Kan complexes.  But at a much simpler level, it allows one to describe sets with operations subject to no equations.  If one wishes to add equations (or operations of greater complexity) then it is natural to consider iterated algebraic injectives, which we introduce and study in the present paper.

Our main application concerns Grothendieck's weak $\omega$-groupoids, introduced in Pursuing Stacks, and the closely related definition of weak $\omega$-category due to Maltsiniotis.  Using $\omega$ iterations we describe these as iterated algebraic injectives and, via this correspondence, prove the faithfulness conjecture of Maltsiniotis.  Through work of Ara, this implies a tight correspondence between the weak $\omega$-categories of Maltsiniotis and those of Batanin/Leinster.
\end{abstract}

\section{Introduction}

The categorical notion of injectivity captures interesting structures in many areas --- for instance, injective modules in algebra and Kan complexes in homotopy theory.  \emph{Algebraic injectivity} was introduced \cite{Garner2011Understanding,Nikolaus2011Algebraic} with the homotopical examples in mind, but even in the simplest settings it has much to say.  

For instance, whilst all sets are injective with respect to monomorphisms, an algebraic injective with respect to the inclusion $j:2 \hookrightarrow 3$ is a set $X$ equipped with a binary operation $m$, or magma, as depicted below.
\begin{equation*}
\cd{
2 \ar[d]_{j} \ar[r]^{(a,b)} & X \\
3 \ar[ur]_{(a,b,m(a,b))}
}
\end{equation*}

Now iterating injectivity is a futile affair: injectives in injectives are injectives.  On the other hand, in the present paper we will show that \emph{iterating algebraic injectivity} produces interesting structures.  For a simple example, all universal algebraic structures are algebraic injectives in algebraic injectives (a.k.a.~\emph{$2$-injectives}) in $\Set$ --- see Example~\ref{example:assmagma} below.

The primary case of interest here concerns the cellular globular theories of \cite{Malts2010} used to define Grothendieck weak $\omega$-groupoids \cite{Grothendieck1983Pursuing}.  In this setting, we will show that models of cellular globular theories can be identified as certain $\omega$-injectives --- iterated algebraic injectives using $\omega$ iterations.  We will use this identification to prove the \emph{faithfulness conjecture} of Maltsiniotis \cite{Malts2010}, which we now describe in a little more detail.

Grothendieck's weak $\omega$-groupoids \cite{Grothendieck1983Pursuing} and the weak $\omega$-categories introduced by Maltsiniotis in \cite{Malts2010} are globular sets with additional structure, expressed in terms of their being models of certain \emph{globular theories}.  In order that the models of a globular theory $\mathbb T$ are equipped with the operations expected in a higher categorical structure one imposes contractibility conditions upon $\mathbb T$.  To capture weakness of these operations one requires that $\mathbb T$ is \emph{cellular} --- this means that it can be constructed as a colimit of a chain 
\begin{equation*}
\xymatrix{
\Theta_{0}=\mathbb T_{0} \ar[r] & \mathbb T_{1} \ar[r] & \cdots \ar[r] & \mathbb T_{n} \ar[r]^{J^{n}_{n+1}} & \mathbb T_{n+1} \ar[r]^{J^{n+1}_{\omega}} & \mathbb T
}
\end{equation*}
in which $\Theta_0$ is the initial globular theory, and in which each $J^{n}_{n+1}:\mathbb T_{n} \to \mathbb T_{n+1}$ is obtained by \emph{freely adjoining fillers} --- see Section~\ref{sect:cellular}.  In \cite{Malts2010} Maltsiniotis conjectured that each functor $J^{n}_{m}:\mathbb T_n \to \mathbb T_m$ for $n < m$ defining a cellular globular theory is faithful.  Assuming this conjecture Ara \cite{Ara2010Thesis} established a sharp correspondence between the weak $\omega$-categories introduced by Maltsiniotis and those of Batanin/Leinster \cite{Batanin1998Monoidal,Leinster2002A-survey}.  However the faithfulness conjecture, on which the correspondence depends, was left unproven.  Using the framework of iterated algebraic injectivity, we prove it in Theorem~\ref{thm:faithfulness}.

Let us now describe the structure of the paper.  In Section~\ref{sect:globular} we describe some background on globular sets, globular theories and the faithfulness conjecture.  In Section~\ref{sect:iterated} we recall algebraic injectivity, introduce iterated algebraic injectivity and the important class of $(\A,\B)$-iterated algebraic injectives, which are parametrised by a set of objects $\A$ and family of maps $\B$ in the base category.  In Theorem~\ref{thm:main} we show that, for a suitable choice of $\A$ and $\B$, these capture the models of cellular globular theories.  In Section~\ref{sect:free} we abstract results of Nikolaus \cite{Nikolaus2011Algebraic} relating to the construction of free algebraic injectives, and extend them to our iterated setting.  Using these results, we prove the faithfulness conjecture in Section~\ref{sect:faithfulness}.  In Section~\ref{sect:nervous}, using the framework of \cite{Bourke2018Monads}, we generalise some of the results in Section~\ref{sect:iterated} away from the globular setting --- in particular, we show that in general $(\A,\B)$-iterated algebraic injectives are the categories of algebras for cellular monads/theories with respect to a class of maps determined by $\A$ and $\B$.   

%

\section*{Acknowledgements}
The main results of this paper were presented at CT2016 in Halifax \cite{BourkeCT2016}.  I am grateful to the organisers for inviting me to give a talk, and for the very enjoyable conference.  I am also grateful to Dimitri Ara and Georges Maltsiniotis for encouraging me to publish the results described herein.  I acknowledge the support of the Grant Agency of the Czech Republic under the grant 19-00902S.


\section{Background on globular theories}\label{sect:globular}

\subsection{Globular sets}

The \emph{globe category} $\mathbb G$ is freely generated by the
    graph$$\cd{
0 \ar@<-5pt>[r]_{\tau_{1}} \ar@<5pt>[r]^{\sigma_{1}} & 1 \ar@<-5pt>[r]_{\tau_{2}} \ar@<5pt>[r]^{\sigma_{2}} & \cdots \ar@<-5pt>[r]_{\tau_{n-1}} \ar@<5pt>[r]^{\sigma_{n-1}} & n-1 \ar@<-5pt>[r]_{\tau_{n}} \ar@<5pt>[r]^{\sigma_{n}} & n \cdots}$$
subject to the relations $\sigma_{n} \circ \sigma_{n-1} = \sigma_{n} \circ \tau_{n-1}$ and $\tau_{n} \circ \sigma_{n-1} = \tau_{n} \circ \tau_{n-1}$.  These relations ensure that there are precisely two maps $\sigma_{n,m}:\tau_{n,m}:n \rightrightarrows m$ for $m > n$.  We abbreviate $\sigma_{n,m}$ and $\tau_{n,m}$ by $\sigma$ and $\tau$ when the context is clear.

The presheaf category $[\mathbb G^\mathrm{op},\Set]$ is the category
    of \emph{globular sets}.  A globular set $A:\f G^{op} \to \Set$ is specified 
  by objects $A(n)$ together with morphisms
$$\cd{
A(n) \ar@<0.8ex>[r]^-{s_{n}} \ar@<-0.8ex>[r]_-{t_{n}} & A(n-1)}
$$
where we write $s_{n}=A(\sigma_{n})$ and $t_{n}=A(\tau_{n})$, or just $s$ and $t$ if the context is clear. 

Elements of the set $A(n)$ are referred to as $n$-cells of $A$.   A pair $(x,y)$ of $n$-cells in $A$ are said to be \emph{parallel} if either $n=0$ or if the equations $s_{n}x=s_{n}y$ and $t_{n}x=t_{n}y$ hold.  A \emph{lifting} for such a parallel pair $(x,y)$ is an element $z \in A(n+1)$ such that $s_{n+1}z=x$ and $t_{n+1}z=y$.  If each parallel pair of $n$-cells in $A$ has a lifting (for each $n$) then the globular set $A$ is said to be \emph{contractible}.

The notions introduced in the preceding paragraph can be understood in terms of lifting properties.  Firstly, an $n$-cell in $A$ amounts to a morphism of globular sets $Yn \to A$ from the representable globular set $Yn=\mathbb G(-,n)$.  Let $j_{n}:S(n) \hookrightarrow Y(n+1)$ be the globular subset of $Y(n+1)$ obtained by omitting the single $n+1$-cell of $Y(n+1)$.  This globular set has two distinct $m$-cells for all $m \leq n$ and none in higher dimensions.  Now a parallel pair of $n$-cells $(x,y)$ in $A$ corresponds to a morphism $S(n) \to A$ and, moreover, the parallel pair admits a lifting just when the corresponding morphism $S(n) \to A$ admits an extension along $j_{n}:S(n) \hookrightarrow Y(n+1)$. 
Later on we will use the following well known expression of $S(n)$ as the coequaliser
\begin{equation*}
\xymatrix{
Y(n-1) + Y(n-1) \ar@/^1ex/[rr]^-{\sigma^{*} + \sigma^{*}} \ar@/_1ex/[rr]_-{{\tau^{*} + \tau^{*}}} && Y(n) + Y(n) \ar[r]^-{\langle j_{n},j_{n} \rangle} & S(n) 
}
\end{equation*}
whose universal property captures the fact that $S(n)$ classifies parallel $n$-cells.

\subsection{The category $\Theta_0$ of globular cardinals}

The category of globular sets has a small dense subcategory
    $\Theta_0$, first described by Berger
    \cite{Berger2002}, whose objects have been termed
    \emph{globular cardinals} by Street \cite{Street2000}.
    These include the representables---the $n$-globes $Yn$ for each
    $n$---but also shapes such as the globular set with distinct cells
    as depicted below.
    
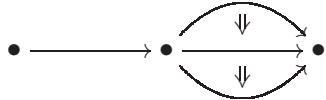
\begin{equation*}      
    \xy
      (0,0)*+{\bullet}="00"; (20,0)*+{\bullet}="10";(40,0)*+{\bullet}="20";
      {\ar^{} "00"; "10"};{\ar@/^1.5pc/^{} "10"; "20"};{\ar@/_1.5pc/^{} "10"; "20"}; {\ar^{} "10"; "20"};{\ar@{=>}^{}(30,6)*+{};(30,1)*+{}};{\ar@{=>}^{}(30,-1)*+{};(30,-6)*+{}};
      \endxy
    \end{equation*}
    The globular cardinals can be parametrised in various ways, for
    instance using trees \cite{Batanin1998Monoidal,
      Berger2002};
    following~\cite{Malts2010}, we will use
    \emph{tables of dimensions}---sequences
    $\vec n=(n_{1}, \ldots ,n_{k})$ of natural numbers of odd length
    with $n_{2i-1} > n_{2i} < n_{2i+1}$. Given such a table $\vec{n}$
    and a functor $D \colon \mathbb G \to \C$, we obtain a diagram
    \begin{equation*}
      \xy
      (12,0)*+{Dn_{2}}="10"; (36,0)*+{Dn_{4}}="30";(60,-5)*+{\cdots}; (84,0)*+{Dn_{k-1}}="40";
      (0,-10)*+{Dn_{1}}="01"; (24,-10)*+{Dn_{3}}="21";(48,-10)*+{Dn_{5}}="41";(72,-10)*+{Dn_{k-2}}="31";(96,-10)*+{Dn_{k}}="51";
      {\ar_{D\tau} "10"; "01"};{\ar^{D\sigma} "10"; "21"};{\ar_{D\tau} "30"; "21"};{\ar^{D\sigma} "30"; "41"};{\ar_{D\tau} "40"; "31"};{\ar^{D\sigma} "40"; "51"};
      \endxy
    \end{equation*}
    whose colimit in $\C$, when it exists, will be written as
    $D(\vec n)$, and called the \emph{$D$-globular sum indexed by
      $\vec n$}. If all such colimits exist then we say that $\C$ admits $D$-globular sums --- let us emphasise that admitting globular sums is not a property of a category but of a category \emph{under} $\mathbb G$.
     Dually a globular object $A:\f G^{op} \to \C$ determines a diagram 
$$
\xy
(12,0)*+{A(n_{2})}="10"; (36,0)*+{A(n_{4})}="30";(60,5)*+{\cdots}; (84,0)*+{A(n_{k-1})}="40";
(0,10)*+{A(n_{1})}="01"; (24,10)*+{A(n_{3})}="21";(48,10)*+{A(n_{5})}="41";(72,10)*+{A(n_{k-2})}="31";(96,10)*+{A(n_{k})}="51";
{\ar_{t} "01"; "10"};{\ar^{s} "21"; "10"};{\ar_{t} "21"; "30"};{\ar^{s} "41"; "30"};{\ar_{t} "31"; "40"};{\ar^{s} "51"; "40"};
\endxy
$$
whose limit, denoted $A(\vec{n})$, is called a \emph{globular product}. 
      
Consider $Y \colon \mathbb G \to [\mathbb G^\mathrm{op},\Set]$.  Taking the full subcategory of $[\f G^{op},\Set]$ on the $Y$-globular sums yields the initial, up to equivalence, category under $\mathbb G$ with globular sums.  $\Theta_{0}$ is a skeleton of this: we can view its objects as the tables of dimensions whilst $\Theta_{0}(\vec{n},\vec{m}) = [\f G^{op},\Set](Y(\vec{n}), Y(\vec{m}))$.  This gives a factorisation 
\begin{equation*}
\xymatrix{
\mathbb G \ar[r]^-{D} & \Theta_0 \ar[r]^-{Y} & [\f G^{op},\Set]
}
\end{equation*}
of the Yoneda embedding, in which $Dn=(n)$ on objects.


\subsection{Globular theories and their models}
A \emph{globular theory} $J:\Theta_0 \to \mathbb T$ is an identity on objects functor preserving globular sums.
The \emph{category $\GTh$ of globular theories} is the full subcategory of $\Theta_{0}/\Cat$ containing the globular theories.

Given a globular theory $J:\Theta_0 \to \mathbb T$ its \emph{category of models} $\Mod(\mathbb T) \hookrightarrow [\mathbb T^{op},\Set]$ is the full subcategory containing the globular product preserving functors.  There is a forgetful functor $U^{T}:\Mod(\mathbb T) \to [\f G^{op},\Set]$ given by restriction along $(J \circ D)^{op}:\mathbb G^{op} \to \mathbb T^{op}$.  
Furthermore a morphism $N:\mathbb S \to \mathbb T$ of globular theories induces, by restriction, a functor $N^{\star}:\Mod(\mathbb T) \to \Mod(\mathbb S)$ commuting with the forgetful functors to the category of globular sets.  

The following proposition records a few basic results that we will need about globular theories and their models.
\begin{Proposition}\label{prop:models}
Let $\mathbb T$ be a globular theory.
\begin{enumerate}
\item The forgetful functor $U^{T}:\Mod(\mathbb T) \to [\f G^{op},\Set]$ is monadic and preserves filtered colimits.  In particular $\Mod(\mathbb T)$ is locally finitely presentable.
\item Each representable $\mathbb T(-,\vec{n}):\mathbb T^{op} \to \Set$ is a model of $\mathbb T$ --- namely the free $T$-model on the globular cardinal $Y\vec{n}$.  Therefore we obtain a factorisation $K_{T}:\mathbb T \to \Mod(\mathbb T)$ of the Yoneda embedding $Y_{T}$ through the full inclusion $I_{T}:\Mod(\mathbb T) \hookrightarrow [\mathbb T^{op},\Set]$.
\item Given a morphism $N:\mathbb S \to \mathbb T$ of globular theories the restriction functor $N^{*}:\Mod(\mathbb T) \to \Mod(\mathbb S)$ has left adjoint $N_{!} = Lan_{K_{S}}(K_{T} \circ N):\Mod(\mathbb S) \to \Mod(\mathbb T)$.
%
\end{enumerate}
\end{Proposition}
\begin{proof}
For monadicity in (1) see Proposition 6.3.6 of \cite{Ara2010Thesis}.  That $U^{T}$ preserves filtered colimits follows easily from the fact that each globular cardinal $Y\vec n$ is finitely presentable as a globular set.  It follows that the resulting monad $T=U^{T}F^{T}$ preserves filtered colimits.  Since by~\cite[Satz~10.3]{Gabriel1971Lokal} the category of algebras for a filtered colimit preserving monad on a locally finitely presentable category is again locally finitely presentable, this completes the proof of (1).

As representables send colimits to limits each representable is a $\mathbb T$-model, and the stated freeness property is simply an instance of the Yoneda lemma.  This establishes (2).  For (3) observe that since $\Mod(\mathbb T)$ is locally presentable, by (1), it is cocomplete.  Therefore we can form the left Kan extension $L=Lan_{K_{S}}(K_{T} \circ N):\Mod(\mathbb S) \to \Mod(\mathbb T)$.  
It remains to prove that $L \dashv N^{\star}$.  Since the extension $L$ is pointwise we have that $Lan_{K_{S}}(K_{T} \circ N) \cong Lan_{Y_{S}}(K_{T} \circ N) \circ \Mod(\mathbb S)(K_{S}-,1)$, where the first component $L^{\prime}$ is the left Kan extension along the Yoneda embedding.  Since the values of $K_{S}$ are representable we have that $\Mod(\mathbb S)(K_{S}-,1) \cong I_{S}$, where $I_{S}$ is the full inclusion viewing models as presheaves.  Hence we have natural isomorphisms $$\Mod(\mathbb T)(LX,A) \cong \Mod(\mathbb T)(L^{\prime}I_{S}X,A) \cong [\mathbb S^{op},\Set](I_{S}X,\Mod(\mathbb T)(K_{T} N-,A))$$
with the second using the Kan adjunction $L^{\prime} \dashv \Mod(\mathbb T)(K_{T} N-,1)$.    Since $N^{\star}$ acts by restriction we have $I_{S} \circ N^{\star} \cong \Mod(\mathbb T)(K_{T} N-,1)$ so that the right side above is naturally isomorphic to $[\mathbb S^{op},\Set](I_{S}X,I_{S}N^{\star}A)$.  By fully faithfulness of $I_{S}$ this in turn is naturally isomorphic to $\Mod(\mathbb S)(X,N^{\star}A)$ and now the composite natural isomorphism $\Mod(\mathbb T)(LX,A) \cong \Mod(\mathbb S)(X,N^{\star}A)$ proves the claim.
\end{proof}


\subsection{Contractible and cellular globular theories}\label{sect:cellular}

Let $J:\Theta_0 \to \mathbb T$ be a globular theory.  A pair of morphisms
$$u,v:(n) \rightrightarrows \vec{m} \in \mathbb T$$
is said to be parallel if either $n=0$ or $u \circ \sigma_n = v \circ \sigma _n$ and $u \circ \tau_n = v \circ \tau_n$.  A \emph{lifting} for the parallel pair consists of a morphism $h:(n+1) \to \vec{m}$ such that $h \circ {\sigma_{n+1}} = u$ and $h \circ {\rho_{n+1}} = v$.  If each parallel pair in $\mathbb T$ has a lifting then $\mathbb T$ is said to be \emph{contractible}.
 
\begin{Remark}\label{rk:contract}
A parallel pair $(u,v)$ in the globular theory $\mathbb T$ as above is simply a parallel pair of $n$-cells in the globular set $\mathbb T(JD-,\vec{m}):\mathbb G^{op}\to \Set$.  In this way we see that $\mathbb T$ is contractible just when the globular set $\mathbb T(JD-,\vec{m}):\mathbb G^{op}\to \Set$ is contractible for each $\vec{m} \in \Theta_0$.  
\end{Remark}

Consider a family $\I$ of parallel pairs $u_i,v_i:(n_i) \rightrightarrows \vec{m}_i$ in the globular theory $\mathbb T$.\begin{footnote}{That is, a set $I$ together with a parallel pair $u_{i},v_{i} \in \mathbb T$ for each $i \in I$.}\end{footnote}
There exists a globular theory  $\mathbb T_{\I}$ and morphism of globular theories $$J_{\I}:\mathbb T \to \mathbb T_{\I}$$ obtained by \emph{freely adjoining liftings} for the parallel pairs of $\I$.  More precisely
\begin{enumerate}
\item For each $i \in I$ the parallel pair $u_i,v_i:(n_i) \rightrightarrows \vec{m}_i \in T_{\I}$ is equipped with a lifting $\phi_{i}:(n_{i}+1) \to \vec{m}_{i}$ and
\item Given a morphism of globular theories $Q:\mathbb T \to \mathbb S$ together with a lifting $\theta_i$ of the parallel pair $(Qu_i,Qv_i)$ for each $i \in I$, there exists a unique morphism of globular theories $\overline{Q}:\mathbb T_{\I} \to \mathbb S$ such that $\overline{Q} \circ J_{\I} = Q$ and $\overline{Q}\phi_i= \theta_i$ for each $i \in I$.
\end{enumerate}
These two properties characterise $\mathbb T_{\I}$ uniquely up to isomorphism under $\mathbb T$. 

A \emph{cellular} globular theory is, by definition, the colimit of a chain of globular theories 
\begin{equation*}
\xymatrix{
\Theta_{0}=\mathbb T_{0} \ar[r]^{J^{0}_{1}} & \mathbb T_{1} \ar[r] & \cdots \ar[r] & \mathbb T_{n} \ar[r]^-{J^{n}_{n+1}} & \mathbb T_{n+1} \ar[r]^-{J^{n+1}_{\omega}} & \mathbb T_{\omega} = \mathbb T
}
\end{equation*}
in which each $J^{n}_{n+1}:\mathbb T_{n} \to \mathbb T_{n+1}$ is of the form $J_{\I_{n+1}}:\mathbb T_n \to \mathbb T_{\I_{n+1}}$ for some family of parallel pairs $\I_{n+1}$ of $\mathbb T_{n}$.  We remark that cellularity in the above sense can be re-expressed as cellularity with respect to a set of maps --- see Example 6.6.  

A \emph{coherator} is a cellular contractible globular theory and, by definition, a \emph{Grothendieck weak $\omega$-groupoid} is a model of a coherator.  The homotopy theory of these structures has been studied in \cite{Malts2010,Ara2013On-the} and more recently in \cite{Henry, Lanari}, though important questions remain open.

Now for $\mathbb T$ cellular we have, on the level of models, a cochain
\begin{equation*}
\cd{ \Mod(\mathbb T) \ar[d]_{U^{\omega}_{n+1}} \ar[dr]^{U^{\omega}_{n}} \ar[drrrr]^{U^{\omega}_{0}} \\
\cdots \Mod(\mathbb T_{n+1}) \ar[r]_<<<<{U^{n+1}_{n}} & \Mod(\mathbb T_{n}) \ar[r] & \cdots \ar[r] & \Mod(\mathbb T_{1})\ar[r]_<<<<{U^{1}_{0}} & \Mod(\Theta_0)}
\end{equation*}
where we write $U^{m}_{n} = (J^{n}_{m})^{*}$ for all $n < m$.  We refer to this as the \emph{tower of models} of $\mathbb T$.  Since taking categories of models sends colimits of globular theories to limits of categories over $Mod(\Theta_0)$  --- see Section 5.3 of \cite{Bourke2018Monads} --- it follows that $\Mod(\mathbb T)$ is the limit of the above cochain of categories.

\begin{ConjectureF}[Maltsiniotis]
Each of the connecting functors $J^{n}_{m}:\mathbb T_{n} \to \mathbb T_{m}$ defining a cellular globular theory $\mathbb T$ is faithful.
\end{ConjectureF}

\begin{Remark}
The conjecture is stated in Section 1.4 of Maltisiniotis \cite{Malts2010} and in the thesis of Ara \cite{Ara2010Thesis} as Conjecture 4.1.7.  It is used in Proposition 6.7.15 and Corollary 6.7.16 of \cite{Ara2010Thesis} to give an explicit description of the globular theory $\Theta_{BL}$ whose models coincide with the weak $\omega$-categories of Batanin-Leinster \cite{Batanin1998Monoidal,Leinster2002A-survey} --- to be precise, the models of the initial globular operad with contraction.
\end{Remark}

Before moving on, let us identify the models of globular theories of the form $\mathbb T_{\I}$.  
To this end, we define a $(\mathbb T, \I)$-model $(X,x)$ to consist of a $\mathbb T$-model $X$ together with a morphism $x_i$ for each $i \in I$ such that the diagram
$$
\cd{X(\vec{m}_i) \ar@<3pt>[dr]^{Xu_i} \ar@<-3pt>[dr]_{Xv_i} \ar[rr]^{x_i} && X(n_i+1) \ar@<3pt>[dl]^{t} \ar@<-3pt>[dl]_{s} \\
& X(n_i)}
$$
is serially commutative.  The morphisms of $\Mod(\mathbb T, \I)$ commute with the given liftings.  There is an evident forgetful functor $U:\Mod(\mathbb T, \I) \to \Mod(\mathbb T)$.

\begin{Proposition}\label{prop:TImodels}
There is an isomorphism of categories $K:\Mod(\mathbb T, \I) \to \Mod(\mathbb T_{\I})$ rendering commutative the diagram 
\begin{equation*}
\xymatrix{
\Mod(\mathbb T_{\I}) \ar[dr]_{{J_{\I}}^{\star}} \ar[rr]^{K} && \Mod(\mathbb T, \I) \ar[dl]^{U} \\
& \Mod(\mathbb T)
}
\end{equation*}
\end{Proposition}
\begin{proof}
The functor $K$ sends a $\mathbb T_{\I}$-model $X$ to the $(\mathbb T, \I)$-model $(X\circ {J_{\I}}^{op},X_{\phi})$ whose first component is $X \circ {J_{\I}}^{op}:\mathbb T^{op} \to \mathbb T_{\I}^{op} \to \Set$ and whose lifting function at $i \in I$ has value $X(\phi_i):X(\vec{m}_i) \to X(n_i+1)$.  Clearly $K$ commutes with the respective functors to $\Mod(\mathbb T)$.  Hence it remains to show that it is bijective on objects and fully faithful.  We establish these properties using endomorphism theories.

To this end, let $X \in \Mod(\mathbb T)$.  By definition, the (relative) endomorphism theory of $X$ is obtained as the (identity on objects/fully faithful)-factorisation of $X^{op}:\mathbb T \to \Set^{op}$ depicted on the left below

\begin{equation*}
\xymatrix{
\mathbb T \ar[dr]_{X^{op}}  \ar[r]^-{E_{X}} & \mathbb End(X) \ar[d]^{M_{X}} && \mathbb T \ar@/^1.85pc/[rr]^{X^{op}}\ar[d]_{F} \ar[r]^-{E_{X}} & \mathbb End(X) \ar[r]^{M_{X}} & \Set^{op}\\
& \Set^{op} && \mathbb S \ar[ur]^{A} \ar[urr]_{B^{op}}
}
\end{equation*}
Since $X^{op}$ preserves globular sums so does $E_{X}$, so that $\mathbb End(X)$ becomes a globular theory on restricting $E_{X}$ along $J:\Theta_{0} \to \mathbb T$ and $E_{X}:\mathbb T \to \mathbb End(X)$ a morphism thereof.  The functor $M_{X}$ also preserves globular sums. The \emph{universal property} of $\mathbb End(X)$ endowed with $M_X$ is that, given a morphism of globular theories $F:\mathbb T \to \mathbb S$, composition with $M_{X}$ induces a bijection between globular theory morphisms $A:\mathbb S \to \mathbb End(X)$ with $A \circ F = E_{X}$ and $\mathbb S$-model structures $B:\mathbb S^{op} \to \Set$ such that $F^{\star}B=X$.  
(This correspondence is indicated in the diagram above right and follows immediately from the fact that $F$, being identity on objects, and $M_{X}$, being fully faithful, are orthogonal.)

Now $\mathbb End(X)(\vec n,\vec m) = \Set(X\vec m,X\vec n)$ with $E_{X}:\mathbb T \to \mathbb End(X)$ acting on morphisms as application of $X$.  It follows that to equip $X$ with the structure $(X,x)$ of a $(\mathbb T,\I)$-model is equally to equip the morphism of theories $E_{X}:\mathbb T \to \mathbb End(X)$ with a lifting $(n_i+1) \to \vec{m}_i$ of the parallel pair $(E_{X}u_i,E_{X}v_i)$ for each $i \in I$.  Therefore by the universal property of $\mathbb T_{\I}$ there exists a unique morphism $\mathbb T_{\I} \to \mathbb End(X)$ extending $E_{X}$ along ${J_{\I}}$; postcomposing with $M_{X}$ gives the corresponding $\mathbb T_{\I}$-model $K^{\prime}(X,x)$.  Spelling out the formulae involved we see that $KK^{\prime}(X,x)=(X,x)$.  To see that $K^{\prime}K=1$ we use the universal property of $\mathbb T_{\I}$.

The proof that $K$ is bijective on arrows is almost identical, except that it uses endomorphism theories of \emph{morphisms} of $\mathbb T$-models.  A $\mathbb T$-model morphism $f:X \to Y$ corresponds to a globular product preserving functor $\mathbb T \to (\Set^{\atwo})^{op}$; now factoring it as (identity on objects/fully faithful) gives $E_{f}:\mathbb T \to \mathbb End(f)$, the (relative) endomorphism theory of $f:X \to Y$.  We leave the details of this part to the reader.
\end{proof}

\nopagebreak

\section{Iterated algebraic injectives and models of cellular theories}\label{sect:iterated}

\subsection{Algebraic injectivity}
By a family $\I = (I,\alpha)$ of morphisms in $\C$ we mean a set $I$ and function $\alpha:I \to Arr(\C)$.

Given such, an \emph{algebraic $\I$-injective} is a pair $(C,c)$ where $C \in \cc$ together with extensions
$$\cd{A_i \ar[d]_{\alpha_i} \ar[r]^{f} & C\\
B_i \ar@{.>}[ur]_{c(i,f)}}$$
for each lifting problem with $i \in I$.  Morphisms $f:(C,c) \to (D,d)$ of algebraic injectives are morphisms $f:C \to D$ commuting with the given extensions. We write $\Ainj(\I)$ for the category of algebraic injectives and $V:\Ainj(\I) \to \C$ for the forgetful functor to $\C$.

\begin{Notation}
We sometimes write $|\I| = \{\alpha_{i}:i \in I\}$ for the underlying set of morphisms in $\C$.  When we assert a property of morphisms of $\I$ --- for example that they are mono --- this should be interpreted as a property of morphisms of $|\I|$.
\end{Notation}

\begin{Example}\label{example:magma}
Let $\iota:2 \hookrightarrow 3 \in \Set$ be the inclusion of the cardinal with $2$ elements into the cardinal with $3$.  It is easy to see that $\Ainj(\iota)$ is the category of magmas.  More generally, consider a finitary signature $\Omega:\mathbb N \to \Set$ and the associated function $I_{\Omega} = \Sigma_{n \in \mathbb N} \Omega(n) \to Arr(\Set)$ sending each element of $\Omega(n)$ to the inclusion  $n \hookrightarrow n+1$.  Then $\Ainj(\I_{\Omega})$ is the category of $\Omega$-algebras and, indeed, categories $\Ainj(\I)$ for families $\I$ of morphisms amongst $\{n \hookrightarrow n+1:n \in \mathbb N\}$ are precisely the categories $\Omega$-Alg of algebras for finitary signatures $\Omega$.
\end{Example}

\begin{Example}
Consider the category of simplicial sets and $\I =\{\Lambda^{k}_{n} \to \Delta_n\}$ the set of horn inclusions.  Then $\Ainj(\I)$ is the category of \emph{algebraic Kan complexes} considered by Nikolaus \cite{Nikolaus2011Algebraic}.
\end{Example}

The following is a standard result.  The case dealing with a set, rather than a family, of morphisms is dealt with in the proof of Theorem 5 of \cite{Bourke2017Equipping} and the generalisation to a family of morphisms is trivial.

\begin{Proposition}\label{prop:lp}
Let $\C$ be a locally presentable category.  Then $\Ainj(\I)$ is locally presentable and $V:\Ainj(\I) \to \C$ an accessible strictly monadic right adjoint.  Furthermore, if $\C$ is locally finitely presentable and $\I$ consists of morphisms with finitely presentable domains then $\Ainj(\I)$ is locally finitely presentable and $V$ finitary.
\end{Proposition}

We now give the key example of algebraic injectivity pertaining to the study of cellular globular theories; namely, that $\f T_{\I}$-models are algebraic injectives in $Mod(\f T)$.  

To see this, let $\I$ be a family of parallel pairs in $\f T$ and consider the parallel pair $u_i,v_i:(n_i) \rightrightarrows \vec{m}_i$ for $i \in I$.  This corresponds to the pair ${u_{i}}_{*},{v_{i}}_{*}:\mathbb T(-,(n_i)) \rightrightarrows \mathbb T(-,\vec{m}_i)$ in $\Mod(\mathbb T)$, which --- by Part 2 of Proposition~\ref{prop:models} --- we can equally write as a pair of maps $${u_{i}}_{*},{v_{i}}_{*}:F(Y(n_i)) \rightrightarrows F(Y(\vec{m}_i))$$ between free $\mathbb T$-models on globular sets.  That $u_i$ and $v_i$ are parallel in $\mathbb T$ corresponds to the fact that we have a fork on the lower row below.
\begin{equation*}
\xymatrix{
F(Y(n_i-1) + Y(n_i-1)) \ar@/^1ex/[rr]^-{F(\sigma^{*} + \sigma^{*})} \ar@/_1ex/[rr]_-{F({\tau^{*} + \tau^{*}})} && F(Yn_i + Yn_i) \ar[dr]_-{F\langle u_{i*},v_{i*} \rangle} \ar[r]^-{F{\langle j_{n_i},j_{n_i} \rangle}} & F(S n_i) \ar[d]^{(u_i,v_i)} \\
&&&     F(Y\vec{m}_i)
}
\end{equation*}
Since the left adjoint $F$ preserves coequalisers, the original pair $u_i,v_i$  therefore corresponds to the factorisation $$(u_i,v_i):F(S n_i) \to F(Y\vec{m}_i)$$ through the coequaliser.  Now to give a $\mathbb T$-model $X$ the structure of a $(\mathbb T,\I)$-model is to give for each $a:\mathbb T(-,\overline{m}_i) \to X$ and $i \in I$ an extension as below left
\begin{equation*}
\xymatrix{
FY(n_i) + FY(n_i) \ar[d]_{\langle \sigma_{*},\tau_{*} \rangle} \ar[rr]^-{\langle {u_{i*},v_{i*} \rangle}} && FY(\vec{m}_i)\ar[d]^{a} && F(S n_i) \ar[d]_{F j_{n_i}} \ar[r]^-{(u_i,v_i)} & FY(\vec{m}_i) \ar[d]^{a} \\
FY(n_i+1) \ar[rr]^-{?} && X && FY(n_i+1) \ar[r]^-{?} & X
}
\end{equation*}
which, equivalently, is to give an extension rendering the rightmost square commutative.   By the universal property of the pushout square
\begin{equation}\label{eq:pushoutnew}
\cd{
F(S{n_i}) \ar[d]_{Fj_{n_i}} \ar[r]^{\langle u_i, v_i \rangle} & FY(\vec{m}_i) \ar[d]^{k_{i}} \\
FY(n_i+1) \ar[r] & FY(\vec{m}_i)_{u_i \simeq v_i}
}
\end{equation}
such are in bijective correspondence with extensions of $a$ along  $k_{i}$ as below.
\begin{equation*}
\cd{
 FY(\vec{m}_i) \ar[d]_{j_{i}} \ar[r]^-{a} & X  \\
 FY(\vec{m}_i)_{u_i \simeq v_i} \ar@{.>}[ur]_{}
 }
 \end{equation*}

Accordingly we see that extensions of $X$ to a $(\f T,\I)$-model are in bijective correspondence with extensions of $X$ to an algebraic $\I^{*}$-injective for the family $\I^{*}$ of morphisms
\begin{equation}\label{eq:indexed}
I \to Arr(\Mod(\mathbb T)):i \mapsto  k_{i}:FY(n_i) \to FY(\vec{m}_i)_{u_i \simeq v_i}
\end{equation}
Extending the above argument to morphisms in the evident manner, we obtain an isomorphism of categories $\Mod(\mathbb T,\I) \cong \Ainj(\I^{*})$ over $\Mod(\mathbb T)$.  Composing this with the isomorphism of Proposition~\ref{prop:TImodels} we obtain one direction of:

\begin{Proposition}\label{prop:isomorphism}
The following coincide up to isomorphism over $\Mod(\mathbb T)$.
\begin{itemize}
\item Forgetful functors $\Mod(\mathbb T_{\I}) \to \Mod(\mathbb T)$ induced by morphisms of theories $\mathbb T \to \mathbb T_{\I}$ for $\I$ a family of parallel pairs in $\mathbb T$;
\item Forgetful functors $\Ainj(\I) \to \Mod(\mathbb T)$ for $\I$ a family of morphisms in $\Mod(\mathbb T)$, each of which is a pushout as in \eqref{eq:pushoutnew}.
\end{itemize}
\end{Proposition}
\begin{proof}
The reverse direction simply involves chasing backwards through the above constructions.
\end{proof}

\subsection{Iterated algebraic injectivity}

Let $I_1\subseteq Arr(\C)$ and $I_2 \subseteq Arr(Inj(\I_1))$ be a set of morphisms in the (non-algebraic) injectivity class $Inj(\I_1)$.  Then $Inj(\I_2)$ is an injectivity class in $\C$ itself --- namely $Inj(\I_1 \cup \I_2)$ --- and so iterating injectivity produces nothing new.  

On the other hand \emph{iterating algebraic injectivity} produces categories that cannot be obtained in a single step.  In fact, the case of interest to us --- Grothendieck weak $\omega$-groupoids --- uses $\omega$ iterations.  We give the definition in its natural generality.

\begin{Definition}
Let $\lambda$ be an ordinal.  By a tower of (algebraic) $\lambda$-injectives over $\C$ we mean a cochain $\Ainj(\I):(\lambda^{+})^{op} \to \CAT$ such that 
\begin{enumerate}
\item $\Ainj(\I_{0})=\C$;
\item For each $n < \lambda$ the link map $V^{n+1}_{n}:\Ainj(\I_{n+1}) \to \Ainj(\I_{n})$ is the forgetful functor from the category of algebraic injectives determined by a given family of maps $\I_{n+1}$ of $\Ainj(\I_{n})$.
\item the cochain is smooth: that is, at a limit ordinal $\gamma \leq \lambda$, we have $\Ainj(\I_{\gamma}) = lim_{n < \gamma} \Ainj(\I_{n})$.
\end{enumerate}
The value of the cochain at $0 < \lambda \in \lambda^{+}$ gives the forgetful functor $V^{\lambda}_{0}:\Ainj(\I_{\lambda}) \to \C$ from the \emph{category of $\lambda$-injectives}.
\end{Definition}

Thus $1$-injectives are algebraic injectives in $\C$ whilst $2$-injectives are algebraic injectives in algebraic injectives, and so on.\begin{footnote}{We use the term $\lambda$-injective as opposed to algebraic $\lambda$-injective since there is no meaningful non-algebraic version of $\lambda$-injectivity beyond the case $\lambda=1$.}\end{footnote}  By the general term iterated algebraic injectives we mean $\lambda$-injectives for some $\lambda$.

As mentioned, the primary structures of interest here are $\omega$-injectives.  A tower of $\omega$-injectives is specified by a diagram

\begin{equation}\label{eq:Iterated}
\cd{ \Ainj(\I)_{\omega} \ar[d]_{V^{\omega}_{n+1}} \ar[dr]^{V^{\omega}_{n}} \ar[drrrr]^{V^{\omega}_{0}} \\
\cdots \Ainj(\I_{n+1}) \ar[r]_<<<<{V^{n+1}_{n}} & \Ainj(\I_{n}) \ar[r] & \cdots \ar[r] & \Ainj(\I_{1}) \ar[r]_<<<<{V^{1}_{0}} & \C}
\end{equation}
in which $\Ainj_{\omega}(\I)$ is the limit of the sequence $\Ainj(\I_{n})_{n < \omega}$.
\begin{Examples}
All locally presentable categories are categories of iterated algebraic injectives in a power of $\Set$.  This is established in the following sequence of examples.
\begin{enumerate}
\item If $\G$ is a directed graph with set of objects $O$ and directed edges $E$ then the presheaf category $\Set^{\G}$ is a category of algebraic injectives $\Ainj(E)$ in $\Set^{O}$ where $E \to Arr(\Set^{O})$ sends $f:n \to m \in E$ to the coproduct inclusion $O(n,-) \to O(n,-) + O(m,-)$.
\item If $\C$ is a small category with underlying graph $U\C$ then it is easy to see that $[\C,\Set]$ is a small orthogonality class in $[U\C,\Set]$.  By Remark 4.4 of \cite{Adamek1994Locally}, each small orthogonality class in $[U\C,\Set]$ is a small injectivity class --- indeed, also a category of algebraic injectives in $[U\C,\Set]$ (since the liftings are, in this case, forced to be unique.)
\item Since each locally presentable category is a small orthogonality class (thus a category of algebraic injectives) in a presheaf category, we conclude --- on combining these examples --- that each locally presentable category is a category of $3$-injectives in a power of $\Set$.
\end{enumerate}
\end{Examples}

Indeed the categories of iterated algebraic injectives in powers of $\Set$ are precisely the locally presentable categories.  If we only allow in each $\I_{n}$ morphisms with finitely presentable domains we get precisely the locally finitely presentable categories.  These claims follows from the following result.

\begin{Proposition}\label{prop:update}
Let $\Ainj(\I):(\lambda^{+})^{op} \to \CAT$ be a tower of $\lambda$-injectives over $\C$.  If $\C$ is locally presentable then each category $\Ainj(\I_{n})$ in the tower is locally presentable and each connecting map $V^{n}_{m}:\Ainj(\I_{n}) \to \Ainj(\I_{m})$ an accessible right adjoint.   (If $\C$ is locally finitely presentable and each $\I_{n}$ consists of morphisms with finitely presentable domains then each $\Ainj(\I_{n})$ is locally finitely presentable and each connecting map a finitary right adjoint.)
\end{Proposition}
\begin{proof}
We prove the statement by transfinite induction on $n \leq \lambda$.  The statement $n=0$ is our assumption that $\C$ is locally presentable.  For $n = m+1$ we have that $\Ainj(\I_{m})$ is locally presentable and $V^{m}_{k}$ an accessible right adjoint for each $k < m$.  By Proposition~\ref{prop:lp} then $\Ainj(\I_{m+1})$ is locally presentable and $V^{m+1}_{m}$ an accessible right adjoint --- since each $V^{m+1}_{k} = V^{m}_{k} \circ V^{m+1}_{m}$ is then a composite of accessible right adjoints the claim holds.  For $n$ a limit ordinal the limit $\Ainj(\I_{n}) = \Ainj(\I_{m})_{m < n}$ is, by Propositions~\ref{prop:FibrantExamples} and \ref{prop:bilimit}, a bilimit.  Since, by Theorem 2.18 of \cite{Bird}, the 2-category of locally presentable categories and accessible right adjoints is closed under bilimits in $\CAT$, the claim follows.  The locally finitely presentable case follows similarly from the locally finitely presentable part of Proposition~\ref{prop:lp} and Theorem 2.17 of \cite{Bird}.
\end{proof}

We write $V^{n}_{m}:\Ainj(\I_{n}) \to \Ainj(\I_{m})$ for the composite forgetful functor where $m < n \in \lambda^{+}$ and $F^{n}_{m} \dashv V^{m}_{n}$ for the adjoints.  In the case $m=0$ we write $V_{n}=V^{n}_{0}:\Ainj(\I_{n}) \to \C$ for the composite and $F_{n}=F^{n}_{0}$ for its left adjoint. 

The examples that we are interested in will arise from a more specialised context, in which we are provided with the data of a triple $(\C,\A, \B)$ consisting of a locally presentable category $\C$, a set $\A \subseteq Ob(\C)$ of objects and a family $\B$ of morphisms of $\C$.

Now consider pushouts in $\Ainj(\I_{n})$ of the form

\begin{equation}\label{eq:pushout}
\cd{
F_{n}A \ar[d]_{F_{n}\alpha} \ar[r]^{u} & F_{n}X \ar[d]^{p}\\
F_{n}B \ar[r]_{q} & P}
\end{equation}
where $\alpha:A \to B \in |\B|$ and $X \in \A$.  

\begin{Definition}
The tower $\Ainj(\I):(\lambda^{+})^{op} \to \CAT$ is a tower of \emph{$(\A,\B)$-iterated algebraic injectives} if 
\begin{quote}
(\dag) for $n < \lambda$ each morphism of $|\I_{n+1}|$  is the pushout of a morphism $F_{n}\alpha:F_{n}A \to F_{n}B$ with $\alpha \in |\B|$ along a morphism $u:F_{n}A \to F_{n}X$ with $X \in A$.
\end{quote}
\end{Definition}

\begin{Proposition}\label{prop:fpobs}
Let $\Ainj(\I):(\lambda^{+})^{op} \to \CAT$ be a tower of $(\A,\B)$-iterated algebraic injectives.  If each object of $\A$ is finitely presentable and 
the family $\B$ consists only of morphisms with finitely presentable domains, then each family $\I_n$ has the same property.  
\end{Proposition}
\begin{proof}
Inspecting the square \eqref{eq:pushout} we see that each morphism in $\I_n$ has source a free $n$-injective on a finitely presentable object.
Therefore it suffices to show that each $F_{n}:\C \to \Ainj(\I_n)$ preserves finitely presentable objects and, for this, it suffices to show that each right adjoint $V_{n}:\Ainj(\I_n) \to \C$ preserves filtered colimits.  This follows by transfinite induction, arguing as in Proposition~\ref{prop:update}.
\end{proof}

\begin{Example}\label{example:assmagma}
Let $\mathbb F \hookrightarrow \Set$ consist of the finite cardinals $0,1,2,\ldots $ and $\B=\{!:0 \to 1,!:2 \to 1\}$.  From Example~\ref{example:magma} $V:\Ainj(\iota:2 \to 3) \to \Set$ is the concrete category of magmas.  Observe that we have a pushout square as below left.
\begin{equation*}
\xymatrix{
0 \ar[d]_{!} \ar[r]^-{} & 2 \ar[d]^{\iota} && F2 \ar[d]_{F!} \ar[r]^-{(t,s)} & F3 \ar[d]^{} \\
1 \ar[r]^-{} & 3 && F1 \ar[r]^-{} & F3/\langle t = s \rangle
}
\end{equation*}
Let $F$ be the left adjoint of $V$.  The terms $t=(xy)z, s=x(yz)$ in $3$ variables are elements of $VF3$.  Together, they correspond to a single map $\langle t,s \rangle: F2 \to F3$. By the universal property of the pushout $F3/\langle t = s \rangle$ each morphism $(a,b,c):F3 \to A$ factors through $F3 \to F3/\langle t = s \rangle$ just when the equation $(ab)c=a(bc)$ holds for all for $a,b,c \in A$.  Thus we see that the category of associative magmas is a category of $(\mathbb F,\B)$-2-injectives.  

Generalising this construction in the obvious way, each equational variety $(\Omega,E)\textnormal{-Alg}$ is naturally a category of $(\mathbb F,\B)$-$2$-injectives.   Conversely, each category of $(\mathbb F,\B)$-iterated algebraic injectives is an equational variety.  One can prove this using sifted colimits and Beck's theorem.  A better argument, not using monadicity theorems, is provided by Theorem~\ref{thm:nervous}.
\end{Example}

Recall the category $\Theta_0$ of globular cardinals, and the set of morphisms $\B=\{S(n) \to Y(n+1):n \in \mathbb N\}$ including the boundaries of representable globular sets.  The main result of this section characterises models of cellular globular theories as $(\Theta_{0},\B)$-$\omega$-injectives.  We will state and prove this result in terms of cochains.

Both the towers of models of a cellular theory and of $\omega$-injectives form $\omega^{+}$-cochains in $\CAT$ where $\omega^{+} = \{0 <1,\ldots < n < \ldots < \omega\}$.  These cochains share the properties of being smooth and \emph{isofibrant}.  Recall that a functor $W:\A \to \B$ is said to be an isofibration if given $A \in \A$ and an isomorphism $f:B \to WA \in \B$ there exists an isomorphism $f^{\prime}:B^{\prime} \to A$ with $Wf^{\prime}=f$.  Now a cochain $X$ is said to be isofibrant if each $X_{n} \to X_{m}$ for $m < n$ is an isofibration of categories.  

By an \emph{equivalence} $F:X \to Y$ of cochains in $\CAT$ we shall mean a natural transformation between cochains for which $F_n:X_n \to Y_n$ is an equivalence of categories for each $n$.  If we restrict our attention to smooth isofibrant cochains then this yields the \emph{correct} notion --- in particular, it yields an \emph{equivalence relation} on smooth isofibrant cochains (see Appendix~\ref{sect:cochains} for a proof).


\begin{Theorem}\label{thm:main}
Up to equivalence of $\omega^{+}$-cochains, the towers of models of cellular globular theories and towers of $(\Theta_0,\B)$-$\omega$-injectives coincide.
\end{Theorem}

\begin{proof}
Consider the tower of models of a cellular globular theory $\mathbb T$.  We will construct the equivalent tower of $\omega$-injectives inductively.  For the base case we use the canonical equivalence $E_{0}=U^{\Theta_0}:\Mod(\mathbb T_0) = \Mod(\Theta_{0}) \to [\mathbb G^{op},\Set]$.  Now suppose that we have a morphism of chains of length $n$ as below
\begin{equation*}\label{eq:sequence}
\cd{
\Mod(\mathbb T_{n}) \ar[d]_{E_{n}} \ar[r] & \cdots \ar[r] & \Mod(\mathbb T_{1}) \ar[r] \ar[d]^{E_{1}} & \Mod(\mathbb T_{0}) \ar[d]^{E_{0}} \\
\Ainj(\I_{n},\alpha_n) \ar[r] & \cdots \ar[r] & \Ainj(\I_{1},\alpha_1) \ar[r] & [\mathbb G^{op},\Set]
}
\end{equation*}
in which each $E_{j}$ is an equivalence and each family $\I_{m}$ satisfies (\dag) for $m \leq n$.  Since $\f T$ is cellular we have $\f T_{n+1}=(\f T_{n})_{\I_{n+1}}$.  Therefore by Proposition~\ref{prop:isomorphism} we have an isomorphism $\Mod(\mathbb T_{n+1}) \to \Ainj(\I_{n+1},j_{n+1})$ making the left triangle below commute.
\begin{equation*}
\cd{
Mod(\f T_{n+1}) \ar[dr]_{U^{n}_{n+1}} \ar[r]^{} & \Ainj(\I_{n+1},j_{n+1}) \ar[d] \ar[r]^{} & \Ainj(\I_{n+1},E_{n}j_{n+1}) \ar[d] \ar[r]^{} & \Ainj(\I_{n+1},\alpha_{n+1}) \ar[dl]^{V^{n}_{n+1}}\\
& Mod(\mathbb T_{n}) \ar[r]_{E_{n}} & \Ainj(K_{n})}
\end{equation*}
Here $\I_{n+1}$ is of the form described in \eqref{eq:indexed}.
Any equivalence of categories $W:\A \to \B$ lifts to an equivalence $\Ainj(S,\alpha) \to \Ainj(S,W\alpha)$ for $(S,\alpha)$ a family of morphisms in $\A$; applying this to $E_{n}$ and $(\I_{n+1},j_{n+1})$ yields the equivalence in the central square above.  The morphisms in the image of $Ej_{n+1}$ do not satisfy (\dag) on the nose --- in particular, they have domains of the form $E_{n}F^{T_{n}}X$ for $X$ a globular cardinal; however, composing these with the natural isomorphism $F_{n}X \cong E_{n}F^{T_{n}}X$ we obtain an isomorphic family $(\I_{n+1},\alpha_{n+1})$ satisfying (\dag) --- this induces an isomorphism as in the triangle above right.  We define $E_{n+1}$ to be the composite equivalence on the top row above.  Now since each $E_{n}$ is an equivalence and since the cofiltered limits $\Mod(\mathbb T)$ and $\Ainj_{\omega}(K)$ are \emph{bilimits} (see Propositions~\ref{prop:FibrantExamples} and \ref{prop:bilimit} of Appendix~\ref{sect:cochains}) the induced map $E:\Mod(\mathbb T) \to \Ainj_{\omega}(K)$ is an equivalence.  This gives the desired equivalence of cochains.  The reverse direction is similar.
\end{proof}

The tower-free version is an immediate consequence.

\begin{Theorem}
Up to equivalence over $[\mathbb G^{op},\Set]$, the categories of models of cellular globular theories and of $(\Theta_{0},\B)$-$\omega$-injectives coincide.  
\end{Theorem}

The above results mean that if we are only interested in the models of cellular globular theories, then we can equally work the semantic notion of $(\Theta_{0},\B)$-iterated algebraic injectives instead.  The following result refines this to deal with cellular contractible theories (aka coherators) whose models are, by definition, the categories of Grothendieck weak $\omega$-groupoids.  

\begin{Theorem}\label{thm:semantic}
Up to equivalence over $[\mathbb G^{op},\Set]$, the categories of models of coherators coincide with those categories of $(\Theta_0,\B)$-$\omega$-injectives having the property that each free $\omega$-injective on a globular cardinal has a contractible underlying globular set.
\end{Theorem}
\begin{proof}
Consider the category of models of a cellular theory $\mathbb T$ and the equivalence $E$ with a category of $(\Theta_{0},\B)$-$\omega$-injectives depicted below.
\begin{equation*}
\xymatrix{
\Mod(\mathbb T) \ar[dr]_{U^{T}} \ar[rr]^{E} && \Ainj_{\omega}(\I) \ar[dl]^{V_{\omega}} \\
& [\mathbb G^{op},\Set]}
\end{equation*}
Since $E$ commutes with the forgetful functors it also commutes with the left adjoints, up to isomorphism.  Thus the free $\omega$-injective on each globular cardinal will be contractible just when the corresponding fact holds for the cellular globular theory $\mathbb T$: the free $\mathbb T$-model on each globular cardinal is contractible.  So it remains to show that this last property holds just when $\mathbb T$ is contractible.  The free $\mathbb T$-model on a globular cardinal $\vec{m}$ is simply $\mathbb T(-,\vec{m}):\mathbb T^{op} \to \Set$ and this has underlying globular set the composite $\mathbb T(J \circ D-,\vec{m}):\mathbb G^{op} \to \Set$.  By Remark~\eqref{rk:contract} each such globular set is contractible just when $\mathbb T$ is so, as required.
\end{proof}

\begin{Remark}
One can tighten the correspondence of Theorem~\ref{thm:main} by working with \emph{concrete models} \cite{Bourke2018Monads} of a globular theory $\mathbb T$.  A concrete model of $\mathbb T$ consists of a pair $(X,A)$ where $X$ is a globular set and $A:\mathbb T^{op} \to \Set$ an extension of $[\mathbb G^{op},\Set](Y-,X):\Theta_{0}^{op} \to \Set$ along $J^{op}$.  The assignment $(X,A) \mapsto A$ forms the object part of an equivalence $\Mod(\mathbb T)_{c} \to \Mod(\mathbb T)$ between the categories of concrete and of ordinary models.  Thus the two categories are essentially the same --- however there is an isomorphism $\Mod(\Theta_0)_{c} \cong [\mathbb G^{op},\Set]$ rather than just an equivalence.  A consequence is that the towers of concrete models of globular theories coincide, up to \emph{isomorphism} of $\omega^{+}$-cochains, with the towers of $(\Theta_0,\B)$-$\omega$-injectives.  
\end{Remark}


\section{Free iterated algebraic injectives and cellularity}

For $\C$ a cofibrantly generated model category with generating trivial cofibrations $\I$ one can define the category of \emph{algebraically fibrant objects} over $\C$ as $V:\Ainj(\I) \to \C$.  In Theorem 2.20 of \cite{Nikolaus2011Algebraic} Nikolaus showed that if 
\begin{itemize}
\item each trivial cofibration is a monomorphism and
\item each morphism in $\I$ has finitely presentable domain
\end{itemize}
then the model structure on $\C$ can be right induced along $V$ to a model structure on $\Ainj(\I)$ for which the adjunction $F \dashv V$ is a Quillen equivalence.\begin{footnote}{In fact, if $\C$ is combinatorial, neither condition is required --- see Theorem 15 of \cite{Bourke2017Equipping}.  However for our present purposes we make use of these conditions and follow Nikolaus' original arguments closely.}\end{footnote}

The main point in establishing the model structure is to show that in pushouts of the form
\begin{equation*}
\cd{
FA \ar[d]_{F\alpha} \ar[r] & X \ar[d]^{f}  \\
FB \ar[r] & Y
}
\end{equation*}
with $\alpha \in \I$ the map $Vf:VX \to VY$ is a weak equivalence.  That the resulting Quillen adjunction is a Quillen equivalence follows on showing that the unit
$$\eta_{X}:X \to VFX$$
is a weak equivalence for all $X$.  Nikolaus achieved these two results by giving explicit, closely related, constructions of such pushouts and of free algebraically fibrant objects.  In fact, Nikolaus did more than this --- showing in fact that the sought for weak equivalences are \emph{trivial cofibrations} --- even \emph{$\I$-cellular maps}.

This last point --- a minor one at first appearance --- obtains particular importance for us.  The reason is that we will not be working with model structures and their weak equivalences, but simply a set of morphisms (or, if one prefers, the cofibrantly generated weak factorisation system they determine).  In our setting we would like to obtain similar free constructions of iterated algebraic injectives and pushouts.  Therefore we begin by abstracting, in Section~\ref{sect:free}, the results of Nikolaus \cite{Nikolaus2011Algebraic} to our setting.  (These abstractions involve little originality since all of the core constructions described below in Section~\ref{sect:free} are in \emph{ibid}.) In Section~\ref{sect:freeIterated} we extend these results to the iterated setting. This leads naturally to a proof of the faithfulness conjecture.

\subsection{Free algebraic injectives and cellularity}\label{sect:free}

Let $\I$ be a family of morphisms in a category $\C$.  Let $Cell(\I)$ denote the class of \emph{$\I$-cellular morphisms}: a morphism is $\I$-cellular if it can be written as a transfinite composite of pushouts of small coproducts of maps in $|\I|$.  It is well known and straightforward to show that $Cell(\I)$ is stable under pushouts, coproducts and transfinite composition.  

In what follows, we work in the context of cocomplete category $\C$ \emph{equipped with a family of morphisms $\I$ with finitely presentable domains such that $Cell(\I) \subseteq Mono$}.  


The constructions of algebraic injectives of interest to us presently are naturally seen as liftings of \emph{sinks} along $V:\Ainj(\I) \to \C$.  Recall that given a set of objects $\{Y_{j}:j \in J\}$ a \emph{sink} under $Y$ consists of an object $X$ and morphisms $\{f_{j}:Y_{j} \to X:j \in J\}$.  Morphisms of the category $Sink(Y)$ of sinks commute with the coprojections from the $Y_{j}$ in the evident manner.  Given a functor $U:\C \to \D$ we obtain a functor 
$$U_{Y}:Sink(Y) \to Sink(UY)$$
given by application of $U$.  By definition, a $U$-\emph{semifinal lifting} of a sink $S=(f_{j}:UY_{j} \to X:j \in J)$ is a representation of the functor $Sink(UY)(S,{U_Y}-):Sink(Y) \to \Set$. Assuming, as is usually done, that $U$ is \emph{faithful} then this amounts to an object $X_{S}$ of $\C$ and morphism $\eta_{S}:X \to UX_{S}$ with the property that each $\eta_{S} \circ f_{j}:UY_{j} \to UX_{S}$ is in the image of $U$, and which is the initial such morphism.

We are interested in two constructions, which are special cases of semifinal liftings along $V:\Ainj(\I) \to \C$ where $|J|=\emptyset$ or $|J|=1$ respectively.  These are
\begin{enumerate}
\item The free algebraic injective on $X \in \C$;
\item The semifinal lifting along $V$ of a \emph{monomorphism} $f:V(Y,y) \to X$.
\end{enumerate}
The first case is the semifinal lifting of the empty sink over $X$ whilst in the second case we view $f$ as a $1$-element sink.  We treat the two cases together by constructing the semifinal lifting of a sink $S=\{f_{j}:V(Y_{j},y_{j}) \to X:j \in J\}$ having $|J| \leq 1$ and each $f_{j}$ a monomorphism.  The semifinal lifting will be the colimit 
\begin{equation*}
\cd{X_{0} \ar[r]^{p^{0}_{1}} & X_{1} \ar[r] & \cdots \ar[r] & X_{n} \ar[r]^{p^{n}_{n+1}} \ar[dr]_{p_{n}} & X_{n+1} \ar[d]^{p_{n+1}} \cdots \\
&&&& X_{S}}
\end{equation*}
of a chain in $\C$.  We start by setting $X_{0}=X$. For each $n$ a subset $$\C_{n}(A,X_{n}) \subseteq \C(A,X_{n})$$ is specified and $X_{n+1}$ is then defined as the pushout below 

\begin{equation*}
\cd{\sum\limits_{i \in I}\C_{n}(A_{i},X_{n}).A_{i} \ar[d]_-{\sum\limits_{i \in I}\C_{n}(A_{i},X_{n}).\alpha_{i}} \ar[rr] && X_{n}  \ar[d]^{p^{n}_{n+1}} \\
 \sum\limits_{i \in I}\C_{n}(A_{i},X_{n}).B_{i} \ar[rr]_<<<<<<{x_{n+1}} && X_{n+1} \hspace{0.2cm} .}
\end{equation*}
Using the coproduct inclusions, we see that $X_{n+1}$ is the \emph{universal object} equipped with a map $p^{n}_{n+1}:X_{n} \to X_{n+1}$ and liftings $x_{n+1}(r,i)$ for each pair $(i \in I,r \in \C_{n}(A_{i},X_{n}))$, as depicted below.
\begin{equation}\label{eq:up}
\xymatrix{
A_{i}\ar[d]_{\alpha_{i}} \ar[rr]^{r} && X_{n}  \ar[d]^{p^{n}_{n+1}} \\
B_{i} \ar[rr]_{x_{n+1}(r,i)} && X_{n+1}}
\end{equation}
Observe also that each connecting map $p^{n}_{n+1}:X_{n} \to X_{n+1}$, is $\I$-cellular, and so a monomorphism.  

Now let $\C_{0}(A,X_{0}) \subseteq \C(A,X_{0})$ consist of those morphisms not factoring through a member $f_{j}$ of the sink.  For higher $n$ we let $\C_{n}(A,X_{n})$ be the set of arrows $A \to X_{n}$ that do not factor through $p^{n-1}_{n}:X_{n-1} \to X_{n}$.  

We must equip $X_{S}$ with the structure of an algebraic injective.  To this end, consider $i \in I$ and $r:A_{i} \to X_{S}$.  There are two cases.

If $r$ factors through $p_{0} \circ f_{j}:Y_{j} \to X \to X_{S}$ as $r^{\prime}$, then this factorisation is unique since both $p_{0}$ and $f_{j}$ are monic.
\begin{equation*}\label{eq:lifting2}
\cd{A_{i} \ar[d]_{\alpha_{i}} \ar[rr]^{r^{\prime}} \ar@/^1.5pc/[rrrr]^{r} && Y_{j} \ar[rr]^{p_{0} \circ f_{j}} && X_{S} \\
B_{i} \ar[urr]_{y_{j}(r^{\prime},i)}}
\end{equation*}
We then define $x_{S}(r,i)=p_0 \circ f_j \circ y(r^{\prime},i)$ as the composite filler depicted.  Note that this definition is forced upon us by the requirement that $p_{0} \circ f_{j}$ be a morphism of algebraic injectives.

Otherwise, since $A_{i}$ is finitely presentable $r:A_{i} \to X_{S}$ factors through $p_{n}:X_{n} \to X_{S}$ for a smallest $n \in \mathbb N$.  Again the factored morphism $r^\prime$ is unique since $p_{n}$ is a monomorphism.  
\begin{equation*}\label{eq:lifting1}
\cd{A_{i} \ar@/^1.5pc/[rrr]^{r} \ar[d]_{\alpha_{i}} \ar[rr]^{r^\prime} && X_{n} \ar[d]_{p^{n}_{n+1}} \ar[r]^{p_{n}} & X_{S} \\
B_{i} \ar[rr]_{x_{n+1}(r^\prime,i)} && X_{n+1} \ar[ur]_{p_{n+1}}}
\end{equation*}
We define $x_{S}(r,i) = p_{n+1} \circ x_{n+1}(r^\prime,i)$ as the composite filler depicted.

\begin{Proposition}\label{prop:semifinal}
Consider a cocomplete category $\C$ equipped with a family of morphisms $\I$ with finitely presentable domains such that $Cell(\I) \subseteq Mono$.  Consider $V:\Ainj(\I) \to \C$.
\begin{enumerate}
\item The forgetful functor $V$ has a left adjoint $F$ and the unit $\eta_{X}:X \to VFX$ is $\I$-cellular for each $X \in \C$.
\item Given a \emph{monomorphism} $f:V(Y,y) \to X$, its semifinal lifting $\eta_{f}:X \to V(X_{f},x_{f})$ exists and is $\I$-cellular.
\end{enumerate}
\end{Proposition}

\begin{proof}
It remains to verify that the morphism $p_{0}:X \to X_{S}=V(X_{S},x_{S})$ has the required universal property.  To this end consider $k:X \to Z=V(Z,z)$ with each $k \circ f_{j}:Y_{j} \to Z$ a morphism of algebraic injectives.  By the universal property of the colimit $X_{S}$, extensions of $k$ as below left
\begin{equation*}
\cd{X_{0} \ar[d]_{p_{0}} \ar[dr]^{k} &&& X_{n} \ar[d]_{p^{n}_{n+1}} \ar[dr]^{k_{n}} \\
X_{S} \ar[r]_{k_{S}} & Z && X_{n+1} \ar[r]_{k_{n+1}} & Z}
\end{equation*} 
are in bijection with families $k_{n}$ satisfying the commutativity above right and having $k_{0}=k$.  

Using the case-by-case definition of the liftings $x_{S}$ it is, moreover, straightforward to see that $k_{S}$ is a morphism of algebraic injectives if and only if each $k \circ f_{j}:Y_{j} \to X \to X_{S}$ is so and when, furthermore, for all $i\in I$, $n \in \mathbb N$ and $r \in \C_{n}(A,X_{n})$ the two composites below from $B_i$ to $Z$ coincide.
\begin{equation*}
\cd{A_i \ar[d]_{\alpha_i} \ar[r]^{r} & X_{n} \ar[r]^{k_{n}}& Z && A_i \ar[d]_{\alpha_i} \ar[rr]^{r} && X_{n} \ar[d]_{p^{n}_{n+1}} \ar[r]^{k_{n}} & Z\\
B_i  \ar[urr]_{z(k_{n} \circ r,i)} &&&& B_i \ar[rr]_{x_{n+1}(r,i)} && X_{n+1} \ar[ur]_{k_{n+1}}}
\end{equation*}
Now by the universal property of $X_{n+1}$ as in \eqref{eq:up} such a $k_{n+1}$ is determined uniquely by $k_{n}$ subject to the above equality --- thus a family $k_{n}$ satisfying the equality in ~\eqref{eq:up} is uniquely determined by the component $k_{0}=k$, and the claim follows.
\end{proof}


\begin{Corollary}\label{cor:pushout2}
Consider a cocomplete category $\C$ equipped with a family of morphisms $\I$ with finitely presentable domains such that $Cell(\I) \subseteq Mono$.  Let $\alpha:A \to B$ be mono and consider $V:\Ainj(\I) \to \C$ with left adjoint $F$.  The pushout
\begin{equation*}
\xymatrix{
FA \ar[d]_{F\alpha} \ar[r]^{} &  X \ar[d]^{f}  \\
FB \ar[r] &  Y}
\end{equation*}
exists and $Vf:V X \to V Y$ is  
a composite of a pushout of $\alpha$ followed by an $\I$-cellular morphism.
\end{Corollary}
\begin{proof}
Given the preceding result, this is a consequence of the following straightforward construction of pushouts in the context of semifinal liftings.  Namely, consider an adjunction $F \dashv U:\C \to \D$ and a span as below left.
\begin{equation*}
\xymatrix{
FA \ar[d]_{F\alpha} \ar[r]^{g} & X && A \ar[r]^{\overline{g}} \ar[d]_{\alpha} & UX \ar[dr]^{Uk_{p}} \ar[d]^{p}   &&& FA \ar[d]_{F\alpha}  \ar[r]^{g} & X \ar[d]^{k_{p}} \\
FB & && B \ar[r]_{q} & Y \ar[r]_{\eta_{p}} & UY_{p} && FB \ar[r]_{\overline{\eta_{p} \circ q}} & Y_{p}\\}
\end{equation*}
Let $Y$ denote the denote the pushout of the corresponding span in $\D$, obtained by transposing through the adjunction, and $Y_{p}$ the semi-final lifting of the pushout coprojection $p:UX \to Y$.  Then $Y_p$ is the pushout of the original span in $\C$, with coprojections as on the right above.
\end{proof}

The following abstracts Theorem 2.20 of \cite{Nikolaus2011Algebraic} away from model categories to the cellular setting.

\begin{Theorem}\label{thm:ordinary}
Let $\C$ be a cocomplete category and $\I$ a family of morphisms in $\C$ with finitely presentable domains such that $Cell(\I) \subseteq Mono$.  Consider the forgetful functor $V:\Ainj(\I) \to \C$ with left adjoint $F$.
\begin{enumerate}
\item Each unit component $\eta_{X}:X \to VFX$ is $\I$-cellular;
\item The forgetful functor $V:\Ainj(\I) \to \C$ sends $F\I$-cellular morphisms to $\I$-cellular morphisms.
\item The left adjoint $F$ is faithful.
\end{enumerate}
\end{Theorem}

It is a special case --- taking $\I=\B$ --- of the result that we require, which is stated and proven below.  This is formulated with iteration in mind.   

\begin{Theorem}\label{thm:ordinary2}
Let $\B$ be a family of morphisms in the cocomplete category $\C$ such that $Cell(\B) \subseteq Mono$.  Let $\I$ be a family of $\B$-cellular morphisms having finitely presentable domains, and consider the forgetful functor $V:\Ainj(\I) \to \C$ with left adjoint $F$.
\begin{enumerate}
\item Each unit component $\eta_{X}:X \to VFX$ is $\B$-cellular;
\item The forgetful functor $V:\Ainj(\I) \to \C$ sends $F\B$-cellular morphisms to $\B$-cellular morphisms;
\item $Cell(F\B) \subseteq Mono$.
\item The left adjoint $F$ is faithful.
\end{enumerate}
\end{Theorem}
\begin{proof}
\begin{enumerate}
\item By Proposition~\ref{prop:semifinal} $\eta_{X}:X\to VFX$ is $\I$-cellular.  Since $\I \subseteq Cell(\B)$, and $Cell(\B)$ is closed under pushouts, coproducts and transfinite composition, we have that $Cell(\I) \subseteq Cell(\B)$, whence $\eta_{X}$ is also $\I$-cellular.
\item By Corollary~\ref{cor:pushout2} each pushout $f$ of a morphism $F\alpha$ in $F\B$ is sent by $V$ to a pushout of $f$ (a $\B$-cellular morphism) followed by an $\I$-cellular, and hence $\B$-cellular, morphism.  Since $\B$-cellular morphisms are closed under composition $Vf$ is $\B$-cellular.  Since the source and target of each $\alpha:A \to B$ in $\I$ is finitely presentable $V$ preserves filtered colimits, and so, transfinite composition.  Thus $V(Cell(F\B)) \subseteq Cell(\B)$.
\item This follows from the preceding part on using that $Cell(\B) \subseteq Mono$ and that $V$, like any faithful functor, reflects monos.
\item  From the first part, we have that each unit component $\eta_{X}:X \to VFX$ is $\B$-cellular and so monic.  We have a commutative triangle
\begin{equation*}
\cd{
\C(Y,X) \ar[drr]_{\eta_{X} \circ -} \ar[rr]^-{F_{Y,X}} && \Ainj(\I)(FY,FX) \ar[d]^{\cong} \\
&& \C(Y,VFX)}
\end{equation*}
whose vertical morphism is the adjointness bijection.  Since $\eta_{X}$ is monic the diagonal function is injective, whence so is the horizontal $F_{Y,X}$, as required.
\end{enumerate}
\end{proof}

\subsection{Free iterated algebraic injectives and cellularity}\label{sect:freeIterated}

The following is our iterated version of the preceding result.  This is specialised to deal with $(\A,\B)$-iterated algebraic injectives in Theorem~\ref{thm:thetaB}.
\begin{Theorem}\label{thm:iterated}
Let $\C$ be locally presentable and $\B$ a family of morphisms in $\C$ such that $Cell(\B) \subseteq Mono$.  Consider a tower $\Ainj(\I)$ of $\omega$-injectives such that for each $n < \omega$ each morphism of $\I_{n+1}$ is an $F_{n}\B$-cellular morphism with finitely presentable domain.
Consider $m<n \in \mathbb \omega^{+}$.
\begin{enumerate}
\item The forgetful functor $V^{n}_{m}:\Ainj(\I_{n}) \to \Ainj(\I_{m})$ sends $F_{n}\B$-cellular morphisms to $F_{m}\B$-cellular morphisms.
\item The unit component $X \to V^{n}_{m}F^{n}_{m}X$ is $F_{m}\B$-cellular for each $X \in \Ainj(\I_m)$.
\item $Cell(F_{n}\B) \subseteq Mono$.
\item The left adjoint $F^{n}_{m}$ is faithful.
\end{enumerate}
\end{Theorem}
\begin{proof}
For finite $n$ the stated properties hold by inductive application of Theorem~\ref{thm:ordinary2} since, by that result, if $Cell(F_{n}\B) \subseteq Mono$ then $Cell(F_{n+1}\B) = Cell(F^{n+1}_{n}F_{n}\B) \subseteq Mono$ too.  It remains to consider the case $n=\omega$ and $m<n$.  In fact the case of general $m$ follows from the case $m=0$ since $\Ainj_{\omega}(\I)$ is equally the limit of the cochain $\Ainj(\I_{k})$ having $k\geq m$.  

Therefore it remains to prove the three claims when $n=\omega$ and $m=0$.  Just as in the proof of Theorem~\ref{thm:ordinary2}, these will follow on showing that given a sink $S=\{f_{j}:V^{\omega}_{0}Y_{j} \to X:j \in J\}$ with at most one member, a mono, its semifinal lifting 
exists and has unit $X \to V^{\omega}_{0} X_{\omega}$ a $\B$-cellular morphism.  We begin by forming the semifinal lifting of the sink along $V^{1}_{0}$, as depicted on the left below. 
\begin{equation*}
\xymatrix{
V^{\omega}_{0}Y_{j} \ar[dr]_{V^{1}_{0}f_{j,1}} \ar[r]^{f_{j}=f_{j,0}} & X \ar[d]^{\eta^{0}_{1}} & & V^{\omega}_{n}Y_{j} \ar[dr]_{V^{n+1}_{n}f_{j,n+1}} \ar[r]^{f_{j,n}} & X_{{n}} \ar[d]^{\eta^{n}_{n+1}}\\
& V^{1}_{0}X_{{1}} &&& V^{n+1}_{n}X_{{n+1}}
}
\end{equation*}
By Theorem~\ref{thm:ordinary} the unit $\eta^{0}_{1}$ is $\B$-cellular.  We now iterate this, forming the $V^{n+1}_{n}$-semifinal lifting of $f_{j,n}$ as on the right above.  Note that this diagram lives in $\Ainj(\I_{n})$.  Arguing inductively, using the established case $n=m+1$, we see that each $f_{j,n}$ is mono and that each $\eta^{n}_{n+1}$ is $F_{n}\B$-cellular.  Now form the colimit of the following chain $(D_{n})_{n \in \mathbb N}$ in $\C$.

\begin{equation*}
\xymatrix{
X \ar[r]^-{\eta^{0}_{1}} & V^{1}_{0}X_{{1}} \ar[r] & \cdots \ar[r] & V^{n}_{0}X_{{n}} \ar[rr]^-{V^{n}_{0}\eta^{n}_{n+1}} \ar[dr]_{\eta_{n}} && V^{n+1}_{0}X_{{n+1}} \ar[dl]^-{\eta_{n+1}} \cdots \\
&&&& col(D)
}
\end{equation*}

Since $\eta^{n}_{n+1}$ is $F_{n}\B$-cellular its underlying map $V^{n+1}_{0}\eta^{n}_{n+1}$ is $\B$-cellular.  Thus the transfinite composite $\eta_{0}:X \to col(D)$ is $\B$-cellular.

Observe that $col(D)$ is equally the colimit of the restricted diagram $(D_{k})_{k \geq n}$, which lies in the image of $V^{n}_{0}$.  Since $V^{n}_{0}$ creates filtered colimits it follows that $col(D)$ obtains the structure of an $n$-injective $col(D)_{n}$, unique such that each cocone projection $\eta_{k}:V^{k}_{n}X_{{k}} \to col(D)_{n}$ is a morphism of $\Ainj(\I_{n})$ for $m>n$.  Since each $V^{n+1}_{n}$ creates filtered colimits $V^{n+1}_{n}col(D)_{n+1}=col(D)_{n}$, so that the sequence $X_{\omega} = (col(D)_{n})_{n \in \mathbb N}$ is an object of $\Ainj(\I_{\omega})$.

Now consider the composite $\eta_{0}:X \to col(D)=V^{\omega}_{0} X_{\omega}$.  Then $\eta_{0} \circ f_{j} = \eta_{n} \circ f_{j,n}$, both of whose components are morphisms of $n$-injectives.  Thus $\eta_{0} \circ f_{j}$ is a morphism of  $n$-injectives for each $n$, and so a morphism of  $\omega$-injectives.  For the universal property, consider $g:X \to V^{\omega}_{0} Z$ having $g \circ f_{j}$ a morphism of $\Ainj(\I_{\omega})$.
\begin{equation*}
\xymatrix{
V^{\omega}_{0}Y_{j} \ar[r]^-{f_{j}} & X \ar[dr]_{g} \ar[r]^-{\eta^{0}_{n}} & V^{n}_{0}X_{{n}} \ar[d]_{g_{n}} \ar[r]^{\eta_{n}} & V^{\omega}_{0} X_{\omega} \ar[dl]^{g_{\omega}} \\
&& V^{\omega}_{0} Z
}
\end{equation*}
Successively applying the universal properties of the semifinal liftings we obtain a unique morphism $g_{n}$ of  $n$-injectives satisfying $g_{n} \eta^{0}_{n} = g_{0}$.  These induce a unique map $g_{\omega}:V^{\omega}_{0} X_{\omega} \to V^{\omega}_{0} Z$ commuting with the $g_{n}$.  Now to say that $g_{\omega}$ is a morphism of  $\omega$-injectives is equivalent to asking that its restriction along each $\eta_{n}$, the $g_{n}$, is a morphism of  $n$-injectives; thus the unique extension $g_{\omega}$ is a morphism of  $\omega$-injectives, as required.
\end{proof}

\begin{Remark}
As a special case of the above we obtain the construction of free  $\omega$-injectives as colimits of chains of the form
\begin{equation*}
\xymatrix{
X \ar[r] & V^{1}_{0}F^{1}_{0}X  \ar[r] & \cdots  \ar[r]^{} & V^{n}_{0}F^{n}_{0}X \ar[rr]^{V^{n}_{0}\eta F^{n}_{0}X} && V^{n+1}_{0}F^{n+1}_{0}X \ar[r] & X_{\omega} \hspace{0.2cm} .
}
\end{equation*}
\end{Remark}

Let us also state the version for $(\A,\B)$-$\omega$-injectives.  This follows immediately from Theorem~\ref{thm:iterated} and Proposition~\ref{prop:fpobs}.

\begin{Theorem}\label{thm:thetaB}
Let $\C$ be locally presentable, $\A$ a set of finitely presentable objects in $\C$ and $\B$ a family of morphisms with finitely presentable domains in $\C$ such that $Cell(\B) \subseteq Mono$.  Given a tower $\Ainj(\I_{n})_{n \in \omega^{+}}$ of $(\A,\B)$-$\omega$-injectives consider $n,m \in \mathbb \omega^{+}$ with $m < n$.
\begin{enumerate}
\item The forgetful functor $V^{n}_{m}:\Ainj(\I_{n}) \to \Ainj(\I_{m})$ sends $F_{n}\B$-cellular morphisms to $F_{m}\B$-cellular morphisms.
\item The unit component $X \to V^{n}_{m}F^{n}_{m}X$ is $F_{m}\B$-cellular for each $X \in \Ainj(\I_m)$.
\item For each $n$ we have $Cell(F_{n}\B) \subseteq Mono$.
\item The left adjoint $F^{n}_{m}$ is faithful.
\end{enumerate}
\end{Theorem}

\section{Cellular globular theories and the faithfulness conjecture}\label{sect:faithfulness}
Consider a cellular globular theory 
\begin{equation*}
\xymatrix{
\Theta_{0}=\mathbb T_{0} \ar[r]^{J^{0}_{1}} & \mathbb T_{1} \ar[r] & \cdots \ar[r] & \mathbb T_{n} \ar[r]^-{J^{n}_{n+1}} & \mathbb T_{n+1} \ar[r]^-{J^{n+1}_{\omega}} & \mathbb T_{\omega} = \mathbb T
}
\end{equation*}
and its associated tower of models
\begin{equation*}
\cd{
\cdots \Mod(\mathbb T_{n+1}) \ar[r]^-{U^{n+1}_{n}} & \Mod(\mathbb T_{n}) \ar[r] & \cdots \ar[r] & \Mod(\mathbb T_{1})\ar[r]^-{U^{1}_{0}} & \Mod(\Theta_0).}
\end{equation*}
Composing with the equivalence $U^{\Theta_0}:\Mod(\Theta_0) \to [\mathbb G^{op},\Set]$ we obtain for $m<n \leq \omega$ a commutative triangle
\begin{equation*}
\xymatrix{
\Mod(\mathbb T_{n}) \ar[dr]_{U^{T_{n}}} \ar[r]^{U^{n}_{m}} & \Mod(\mathbb T_{m}) \ar[d]^{U^{T_{m}}} \\
& [\mathbb G^{op},\Set] \hspace{0.5cm} .
}
\end{equation*} 
We write $F^{n}_{m} = (J^m_n)_{!}:\Mod(\mathbb T_{m}) \to \Mod(\mathbb T_{n})$ for the left adjoint to the forgetful functor and also $F^{n}_{0}=F^{n}$.

\begin{Theorem}\label{thm:cellulartheories}
Given a cellular globular theory $\mathbb T$ as above
\begin{enumerate}
\item The forgetful functor $U^{n}_{m}:\Mod(\mathbb T_{n}) \to \Mod(\mathbb T_{m})$ sends $F^{n}\B$-cellular morphisms to $F^{m}\B$-cellular morphisms.
\item The unit component $X \to U^{n}_{m}F^{n}_{m}X$ is $F^{m}\B$-cellular for each $X \in \Mod(\mathbb T_{m})$.
\item $Cell(F^{n}\B) \subseteq Mono$.
\item The left adjoint $F^{n}_{m}$ is faithful.
\end{enumerate}
\end{Theorem}
\begin{proof}
By Theorem~\ref{thm:main} there is a $(\Theta_0,\B)$-tower of $\omega$-injectives and cochain equivalence as depicted below.
\begin{equation*}
\xymatrix{
\Mod(\mathbb T_{n}) \ar[d]^{E_{n}} \ar[r]^{U^{n}_{m}} & \Mod(\mathbb T_{m}) \ar[dr]^{U^{T_{m}}} \ar[r]^{U^{m}_{0}} \ar[d]^{E_{m}} & \Mod(\Theta_{0}) \ar[d]^{E_{0}}  \\
\Ainj(\I_{n}) \ar[r]_{V_{n,m}}  & \Ainj(\I_{m}) \ar[r]_{V_{m}} &  [\mathbb G^{op},\Set] \\
}
\end{equation*} 
By virtue of the cochain equivalence, we must show that the corresponding facts to (1),(2) and (3) hold for the tower of $\omega$-injectives.  This is precisely the content of Theorem~\ref{thm:thetaB} on observing that its hypotheses ---  each object of $\Theta_0$ is finitely presentable as a globular set, each morphism of $\B$ has finitely presentable domain and codomain and $Cell(\B)=Mono$ --- are satisfied.
\end{proof}

\begin{Theorem}[The faithfulness conjecture]\label{thm:faithfulness}
Each connecting map $J^{m}_n:\mathbb T_{m} \to \mathbb T_{n}$ defining a cellular globular theory $\mathbb T$ is faithful.  
\end{Theorem}
\begin{proof}
By Proposition~\ref{prop:models} we have $F^n_m = Lan_{K_{T_{m}}}(K_{T_n} \circ J^{m}_{n}):\Mod(\mathbb T_{m}) \to \Mod(\mathbb T_{n})$ where $K_{T_{m}}$ and $K_{T_{n}}$ are the restricted Yoneda embeddings for $\mathbb T_{m}$ and $\mathbb T_{n}$ respectively.  Since $K_{T_{m}}$ is fully faithful it follows that we have a natural isomorphism in the square below.
$$\xy
(0,0)*+{\mathbb T_{m}}="00";(30,0)*+{\mathbb T_{n}}="10";(0,-15)*+{\Mod(\mathbb T_{m})}="01";(30,-15)*+{\Mod(\mathbb T_{n})}="11";
{\ar^{J^{m}_{n}} "00"; "10"};{\ar_{K_{T_{m}}} "00"; "01"};{\ar^{K_{T_{n}}}"10";"11"};{\ar_{F^n_m}"01";"11"};
(15,-7)*+{\cong};
\endxy$$
Now both vertical morphisms are fully faithful whilst $F^n_m$ is faithful by Theorem~\ref{thm:cellulartheories}.  Therefore $J^{m}_{n}$ is faithful too.

\end{proof}

\section{Cellular monads and cellular theories}\label{sect:nervous}

In the present section, we briefly describe a generalisation of Theorem~\ref{thm:main}, which captures the models of cellular globular theories as $\omega$-injectives.  We will present our results in the general setting of \cite{Bourke2018Monads}.  In particular, our base $\E$ will be \emph{enriched} in a symmetric monoidal locally presentable category $\V$.  

Given a family $\I$ of morphisms in $\E$, we can form the \emph{enriched category of algebraic injectives}, an object of which consists of $E \in \E$ together with a section 
$e(i,-):\E(X_{i},E) \to \E(Y_{i},E)$ of $\E(\alpha_{i},E):\E(Y_{i},E) \to \E(X_{i},E)$ for each $i \in I$.  As an enriched category $\Ainj(\I)$ can be encoded as the pullback:
\begin{equation*}
\cd{\Ainj(\I) \ar[d]_{V} \ar[r]^{} & \textnormal{SE}[\I,\V] \ar[d]^{} \\
\E \ar[r]^-{K} & \textnormal{Arr}[\I,\V] \hspace{0.2cm} .}
\end{equation*}
Here the right leg $\textnormal{SE}[\I,\V] \to \textnormal{Arr}[\I,\V]$ is the forgetful functor from the enriched category of \emph{split epimorphisms} in $[\I,\V]$ to the enriched category of arrows therein, and $K$ the $\V$-functor sending $C$ to the family $(\E(\alpha_{i},E):\E(Y_{i},E) \to \E(X_{i},E))_{i \in I}$.  

\begin{Example}\label{example:magmoidalcategories}
Consider $\V=\Cat$ and $\E=\Cat$ as a 2-category.  Building on Example~\ref{example:magma}, let $\iota:2 \to 3$ be the inclusion of the discrete category on two objects to the discrete category on three.  An object of $\Ainj(\iota)$ is a category $C$ equipped with a functor $C^{2} \to C$ --- in particular, $\Ainj(\iota)$ is the 2-category of magmas in $\Cat$.
\end{Example}

Now $\E$ itself will be a locally presentable $\V$-category, which comes equipped with a small dense full sub $\V$-category $\A \hookrightarrow \E$.  This is the basic setting of \cite{Bourke2018Monads}.  It is not hard to see --- for instance, using Lemma 4 of \emph{ibid.}~ --- that $\Ainj(\I)$ is itself a locally presentable enriched category, whose enriched forgetful functor $V$ to $\E$ is itself a right adjoint.  

In this setting one also has a good notion of $\A$-theory --- capturing the globular theories when $\A = \Theta_0$ --- and of $\A$-\emph{nervous} monad.  The two notions are equivalent by Theorem 17 of \emph{ibid}.  In the present section we prove our results using monads because their relationship with signatures, of central importance here, is a little cleaner to express.  We also consider these results from the perspective of $\A$-theories at the end of this section.

Following \emph{ibid.}~ the category $\sig$ of \emph{signatures} is the \emph{ordinary} category $[ob \A,\E]$.  Given $\Omega \in \sig$ an $\Omega$-algebra consists of an object $E \in \E$ together with a morphism $\E(A,E) \to \E(\Omega A,E)$ for each $A \in \A$.  $\Omega$-algebras, together with the natural structure on homs, form an enriched category $Alg(\Omega)$.  Now there is a forgetful functor $U:\mnd \to \sig$ sending an enriched monad $T$ to the family $(TA)_{A \in \A}$.  By Theorem 36 of \cite{Bourke2018Monads} this has a left adjoint $F$.  Furthermore by the proof of Proposition 53 of \emph{ibid}, the $\V$-category of $F\Omega$-algebras is, up to isomorphism over $\E$, the $\V$-category $Alg(\Omega)$ just described.

Now although the category of monads is not typically cocomplete it does admit all colimits of free monads on signatures.  In fact the colimit closure $\nerv$ of such free monads is extremely well behaved: it is locally presentable, monadic over $\sig$ and contains exactly the so-called \emph{$\A$-nervous monads.}  For more on nervous monads and proofs of the above claims see \cite{Bourke2018Monads}.

The key signatures for us are the following ones.  Namely, given a pair $(A \in \A,X \in \E)$ we define the signature $A | {X}:ob \A \to \E$ to have value $X$ at $A$ and the initial object $\emptyset$ otherwise.  Thus a $A | {X}$-algebra consists of an object $E$ together with a single morphism $\E(A,E) \to \E(X,E)$.  At $f:X \to Y$ we have the map $A | {f}:A | {X} \to {A} | {Y}$ with value $f$ at $A$.

Evidently there is a natural isomorphism 
\begin{equation*}\label{eq:sig}
\sig(A|{X},\Omega) \cong \E(X,\Omega(A)) \hspace{0.2cm} .
\end{equation*}
Given the morphism of signatures $A|{f}:A|{X} \to {A}|{Y}$ we can form the corresponding morphism $F(A|{f})$ of monads; now a pushout of a coproduct of such maps
\begin{equation*}
\xymatrix{
 F(\Sigma_{i \in I} {A_{i}}|{X_{i}}) \ar[d]_{F(\Sigma_{i \in I} {A_{i}}|{f_{i}})} \ar[r]^-{\langle t_{i} \rangle_{i \in I}} & T \ar[d]^{P_{\I}} \\
 F(\Sigma_{i \in I} {A_{i}}|{Y_{i}}) \ar[r] & T_{\I}
}
\end{equation*}
is specified by a family of triples $(f_{i}:X_{i} \to Y_{i},t_{i}:F ({A_{i}}|{X_{i}}) \to T, A_{i})_{i \in I}$.  What are the algebras for the pushout $T_{\I}$?  Since, by Proposition 28 of \emph{ibid}, semantics sends colimits to limits we have a pullback square as on the left below.
\begin{equation*}
\xymatrix{
\Alg(T_{\I}) \ar[rr] \ar[d]_{\Alg(P_{\I})} && \Alg(F(\Sigma_{i \in I} {A_{i}}|{Y_{i}})) \ar[d] \ar[rr] && \Alg(\Sigma_{i \in I} {A_{i}}|{Y_{i}}) \ar[d]^{\Alg(\Sigma_{i \in I}{A_{i}}|f_{i})} \\
\Alg(T) \ar[rr] && \Alg(F(\Sigma_{i \in I} {A_{i}}|Y_{i})) \ar[rr] && \Alg(\Sigma_{i \in I} {A_{i}}|{X_{i}})
}
\end{equation*}
Since we have horizontal isomorphisms in the right square above, the outer square is a pullback.   Now under the isomorphisms
\begin{equation*}
\mnd(F({A_{i}}|{X_{i}}),T) \cong \sig({A_{i}}|{X_{i}},UT) \cong \E(X_i,TA_{i})
\end{equation*}
the map $t:F({A_{i}}|{X_{i}}) \to T$ bijectively corresponds to a morphism $\overline{t}: X_{i} \to TA_{i}$.  Furthermore a $\Sigma_{i \in I}{A_{i}}|{X_{i}}$-algebra on $E \in \E$ is specified by an ${A_{i}}|{X_{i}}$-algebra structure on $E$ for each $i \in I$.  In these terms, the horizontal composite functor on the bottom row of the above diagram sends the $T$-algebra $(E,e)$ to the $\Sigma_{i \in I} {A_{i}}|{X_{i}}$-algebra whose value at $i \in I$ is the morphism 
\begin{equation*}
\xymatrix{
\E(A_i,E) \ar[r]^-{T} &\E(TA_i,TE) \ar[r]^-{\E(\overline{t_i},e)} & \E(X_i,E) \hspace{0.5cm} .
}
\end{equation*}
Therefore a $T_{\I}$-algebra amounts to a $T$-algebra $(E,e)$, together with a lifting
\begin{equation}\label{eq:lifting}
\xymatrix{
\E(A_i,E) \ar[d]_{T} \ar[r]^{?} & \E(Y_i,E) \ar[d]^{\E(f_i,E)} \\
\E(TA_i,TE) \ar[r]_{\E(\overline{t_i},e)} & \E(X_i,E)
}
\end{equation}
for each $i \in I$.  Morphisms of $T_{\I}$-algebras amount to morphisms of $T$-algebras commuting with such liftings in the evident sense, and they form the morphisms of a $\V$-enriched category in a natural way.

Now transposing $\overline{t}$ through the isomorphism $\E(X_i,TA_i) \cong \Alg(T)(F^{T}X_i,F^{T}A_i)$ to a map $t_i^{*}$ we may form the following pushout 
\begin{equation}\label{eq:pushoutnew2}
\xymatrix{
F^{T}X_i \ar[d]_{F^{T}f_i} \ar[r]^-{t_i^{*}} & F^{T}A_i \ar[d]^{p_i} \\
F^{T}Y_i \ar[r] & P_{i}
}
\end{equation}
in the enriched category of $T$-algebras.  We thereby obtain a corresponding family $(p_{i})_{i \in I}$
of morphisms in $\Alg(T)$.  Using the universal property of the pushout we see that to equip $(E,e)$ with the structure of an algebraic $\I$-injective is equally to give a lifting as on the left below
\begin{equation*}
\xymatrix{
\Alg(T)(F^{T}A_i,(E,e)) \ar[dr]_--{\Alg(T)(t_i^{*},(E,e))} \ar[r]^{?} & \Alg(T)(F^{T}Y_i,(E,e)) \ar[d]^{\Alg(T)(F^{T}f_i,1)} \\
& \Alg(T)(F^{T}X_i,(E,e))
}
\hspace{0.5cm}
\xymatrix{
\E(A_i,E) \ar[dr]_{\E(\overline{t_i},e) \circ T} \ar[r]^{?} & \E(Y_i,E) \ar[d]^{\E(f_i,1)} \\
& \E(X_i,E)
}
\end{equation*}
for each member of $\I$.  Transposing through the adjunction $F^{T} \dashv U^{T}$ and using that the transpose of $t_i^{*}$ is $\E(\overline{t_i},e) \circ T$ as in \eqref{eq:lifting} we conclude that this is precisely to equip $(E,e)$ with the structure of a $T_{\I}$-algebra.  With a similar extension to hom-objects we obtain one direction of:

\begin{Proposition}\label{prop:generalisation}
The following coincide up to isomorphism over $\Alg(T)$.
\begin{itemize}
\item Forgetful functors $\Alg( T_{\I}) \to \Alg(T)$ induced by morphisms of $ T \to  T_{\I}$ where $T \in \nerv$.
\item Forgetful functors $\Ainj(\I) \to \Alg(T)$ for $\I$ a family of morphisms in $\Alg(T)$, each of which is a pushout as in \eqref{eq:pushoutnew2}.
\end{itemize}
\end{Proposition}
\begin{proof}
The reverse direction simply involves chasing backwards through the above constructions.
\end{proof}

Given a family $\B$ of morphisms of $\E$ let $$F(\A|{\B})=(F(A|{f_i}):F(A|{X_i}) \to F({A}|{Y_i}))_{(A,i) \in \A \times \B}$$ be the induced family in $\nerv$.  By definition, a morphism in $\nerv$ is $F(\A|{\B})$-$\lambda$-cellular for an ordinal $\lambda$ if it is a $\lambda$-composite of pushouts of coproducts of maps in $F(\A|{\B})$; an object is $\lambda$-cellular if the unique map to it from the initial monad $\textnormal{Id}_{\E}$ is $\lambda$-cellular.  Accordingly, on taking $\V$-categories of algebras, a $\lambda$-cellular monad gives rise to a $\lambda^{+}$-cochain of $\V$-categories and $\V$-functors, with base $\Alg{(\textnormal{Id}_{\E})} \cong \E$.

\begin{Theorem}\label{thm:nervous}
Let $\A \hookrightarrow \E$ be a small dense full subcategory of a locally presentable category and $\B$ a family of morphisms of $\E$.  Given an ordinal $\lambda$, towers of $(\A,\B)$-$\lambda$-injectives coincide, up to isomorphism of $\lambda^{+}$-cochains, with towers of algebras of $F(\A|{\B})$-$\lambda$-cellular monads.\end{Theorem}
\begin{proof}
Given Proposition~\ref{prop:generalisation} the proof follows inductively, arguing as in \ref{thm:main}. 
\end{proof}

\begin{Examples}
Consider again $\E= \V=\Cat$, and $\mathbb F \hookrightarrow \Cat$ the dense full sub-2-category of finite cardinals, viewed as finite discrete categories and let 
$$\B=\{j_{0}:\varnothing \to (\bullet),j_{1}:(\bullet \hspace{0.2cm} \bullet) \to (\bullet \to \bullet), j_{2}:(\bullet \rightrightarrows \bullet) \to (\bullet \to \bullet)\}$$
 be the generating cofibrations for the folk model structure on $\Cat$ \cite{Lack2007Homotopy-theoretic}.  
  Then structures such as monoidal categories --- whose definition involves \emph{no equations between objects} --- can be described as $(\mathbb F,\B)$-iterated algebraic injectives in $\Cat$.  
  
For a simple case, building on Example~\ref{example:assmagma}, let us describe categories $C$ equipped with a tensor product $m:C^{2} \to C:(x,y) \mapsto xy$ and an associator $\alpha:(xy)z \to x(yz)$ satisfying MacLane's pentagon equation as $3$-injectives in $\Cat$.  From Example~\ref{example:magmoidalcategories} objects of $\Ainj(i_{0})$ are categories equipped with a tensor product.  The terms $t=(xy)z, s=x(yz)$ in three variables are elements of $UF3$  and, together, correspond to a single map $\langle t,s \rangle: F_{1}(\bullet \hspace{0.2cm} \bullet) \to F_{1}3$.  Forming the pushout as in the second square below, we see that to equip $(C,m)$ with the structure of an object of $\Ainj(i_{1})$ is to give a natural transformation $\alpha:(xy)z \to x(yz)$.  
\begin{equation*}
\xymatrix{
\varnothing \ar[d]_{j_0} \ar[r]^-{} & 2 \ar[d]^{i_{0}} & F_{1}(\bullet \hspace{0.2cm} \bullet) \ar[d]_{F_{1}j_{1}} \ar[r]^-{(t,s)} & F_{1}3 \ar[d]^{i_{1}} & F_{2}(\bullet \rightrightarrows \bullet) \ar[d]_{F_{2}j_{2}} \ar[r]^-{(u,v:l \rightrightarrows r)} & F_{2}4 \ar[d]^{i_{2}}\\
(\bullet) \ar[r]^-{} & 3 & F_{1}(\bullet \to \bullet) \ar[r]^-{} & F_{1}3/\langle t \to s \rangle & F_{2}(\bullet \to \bullet) \ar[r]^-{} & F_{2}4/\langle u = v \rangle
}
\end{equation*}
Now let $l=((wx)y)z$ and $r=w(x(yz))$.  The two paths $u,v:l\rightrightarrows r$ of the pentagon specify a parallel pair in $U_{2}F_{2}4$ and so correspond to a single map $F_{2}(\bullet \rightrightarrows \bullet) \to F_{2}4$.  Forming the pushout in the third square below, we see that  $(C,m,\alpha)$ admits the structure of an object of $\Ainj(i_{2})$ precisely if the pentagon equation for $\alpha$ holds.

Building on this example, we can use $(\mathbb F,\B)$-$3$-injectivity to capture any monadic algebraic structure borne by categories, whose basic operations are of the form $C^{n} \to C$ for finite cardinals $n$, and whose defining equations only involving equalities between morphisms rather than objects.  A precise characterisation of the 2-monads describing such categorical structures was given in Section 6 of \cite{Bourke2013On-semiflexible}, where they were called the \emph{pie presentable} strongly finitary 2-monads.  Indeed using Theorem 34 of \emph{ibid}, Theorem~\ref{thm:nervous} above and the pseudo-variant of Theorem 17 of \cite{Bourke2016Awfs2}, one can show that $(\mathbb F,\B)$-$\lambda$-injectives are $(\mathbb F,\B)$-$3$-injectives for all higher ordinals $\lambda$, with both classes coinciding as exactly the algebras for pie presentable strongly finitary 2-monads. 

\end{Examples}

As mentioned already, Theorem~\ref{thm:nervous} can be reformulated using $\A$-theories rather than monads.  By Theorem 17 of \cite{Bourke2018Monads} we have an equivalence of categories $\Psi:\th \cong \nerv$.  Translating through this equivalence we obtain a forgetful functor $U:\th \to \sig$ with left adjoint $F$, abusively using the same symbols $U$ and $F$ as before.  Since, by Theorem 24 of \emph{ibid}, $\Psi$ commutes with semantics in the sense that $\Mod(\mathbb T) \simeq \Alg(\Psi\mathbb T)$, it follows that towers of algebras of $F(\A|\B)$-$\lambda$-cellular monads coincide, to within equivalence, with towers of models of $F(\A|\B)$-$\lambda$-cellular theories.  Thus we may reformulate Theorem~\ref{thm:nervous} using the language of $\A$-theories, as follows.

\begin{Theorem}\label{thm:theories}
Let $\A \hookrightarrow \E$ be a small dense full subcategory of a locally presentable category and $\B$ a family of morphisms of $\E$.  Given an ordinal $\lambda$, towers of $(\A,\B)$-$\lambda$-injectives coincide, up to equivalence of $\lambda$-cochains, with towers of models of $F(\A|{\B})$-$\lambda$-cellular theories.
\end{Theorem}

\begin{Example}
In Section 3.11 of \cite{Malts2010} Maltsiniotis shows that the cellular globular theories, in the sense of Section~\ref{sect:cellular}, are exactly the $\omega$-cellular theories relative to a set of maps
$$
\I = \{\Theta_{0}[\vec m,n] \to \Theta_{0}[\vec m,n]^{\prime} \in \GTh: \vec m \in \Theta_0,n \in \mathbb N\} \hspace{0.2cm} .
$$
These globular theories are characterised by universal properties --- a morphism $\Theta_{0}[\vec m,n] \to \mathbb T$ amounts to a parallel pair $(n) \rightrightarrows \vec{m}$ in $\mathbb T$, whilst a morphism from $\Theta_{0}[\vec m,n]^{'}$ to $\mathbb T$ amounts to a morphism $(n+1) \to \vec{m}$.  Furthermore $\Theta_{0}[\vec m,n] \to \Theta_{0}[\vec m,n]^{\prime}$ is the evident map determined by these universal properties.

To explain this using  the framework of the present section, we work in ordinary $\Set$-enriched category theory, with base $\E = [\mathbb G^{op},\Set]$ and $\A=\Theta_0$.   Then the $\Theta_0$-theories are --- by Examples 42(v) of \cite{Bourke2018Monads} --- precisely the globular theories.  Moreover, the forgetful functor $U:\GTh \to [ob \Theta_0,[\mathbb G^{op},\Set]]$ to signatures --- to within isomorphism --- sends $J:\Theta_0 \to \mathbb T$ to the signature $\vec{m} \mapsto \mathbb \mathbb T(JD-,\vec{m})$.

Now consider the set $\B = \{ j_n:S(n) \to Y(n+1): n \in \mathbb N\}$ of boundary inclusions.  To give a morphism of globular theories $F(\vec{m}|{S(n)}) \to \mathbb T$ is to give a morphism of signatures $\vec{m}|{S(n)} \to UT$.  In turn, this amounts to a morphism $S(n) \to T(JD-,\vec{m})$; that is, a parallel pair $(n) \rightrightarrows \vec{m} \in T$.  Therefore $F(\vec{m}|{S(n)}) \cong \Theta_{0}[\vec m,n]$.  Continuing in this way we see, indeed, that the set $\I$ considered by Maltisiniotis is nothing but our canonical set of morphisms $F({\Theta_{0}}|{\B})$.  Given this identification, Theorem~\ref{thm:main} becomes  a special case of Theorem~\ref{thm:theories} above.
\end{Example}

\appendix

\section{Fibrant cochains}\label{sect:cochains}

In the present section we prove the results about fibrant cochains needed within the paper.

\begin{Proposition}\label{prop:FibrantExamples}
Towers of models of cellular globular theories and of $\lambda$-injectives are smooth and isofibrant.\end{Proposition}
\begin{proof}
Smoothness in the first case was observed in Section~\ref{sect:cellular}, whilst in the second case is by definition.  Since isofibrations are defined by a right lifting property, they are stable under transfinite cocomposition.  Given this, a smooth cochain $X$ is isofibrant just when each link map $X_{n+1} \to X_{n}$ is an isofibration.  
Both classes of tower being smooth, it suffices to show that if $J:\mathbb T \to \mathbb S$ is a morphism of globular theories then $J^{*}:\Mod(\mathbb S) \to \Mod(\mathbb T)$ and likewise that each forgetful functor of the form $V:\Ainj(\I) \to \C$ is an isofibration.  

Since the structure of an algebraically injective object can be uniquely transferred along an isomorphism, the latter case is clear.  As for the former, since each category of models is a \emph{replete} full subcategory of the corresponding presheaf category, it suffices to show that $[J^{op},1]:[\mathbb S^{op},\Set] \to [\mathbb T^{op},\Set]$ is an isofibration.  This follows easily from the fact that $J$ is bijective on objects.
\end{proof}

The following results about the good properties of smooth isofibrant cochains are naturally 2-categorical --- as such, we state and prove them using general 2-categories rather than just $\CAT$.  Here a 2-category refers to a \emph{strict} 2-category.  Given 2-categories $\A$ and $\B$ the collection of strict 2-functors, pseudonatural transformations and modifications between them forms a 2-category $Ps(\A,\B)$.  Limits and isofibrations are representable notions in a 2-category --- thus, a cochain $(X_{n})_{n \in \lambda} \in \C$ is said to be smooth/isofibrant if for each $C \in \C$ the induced cochain $\C(C,X_{n})_{n \in \lambda}$ of categories is smooth/isofibrant.  

\begin{Lemma}
Consider a 2-category $\C$ and a pseudonatural transformation $f:X \rightsquigarrow Y$ of $\lambda$-cochains in $\C$.  If $Y$ is smooth and isofibrant, then there exists a 2-natural transformation $g:X \to Y$ and invertible modification $\phi:f \cong g$.
\end{Lemma}
\begin{proof}
If $\lambda = 0$ then it is trivial.  For $\lambda$ non-zero, we define $g_{0}=f_{0}:X_{0} \to Y_{0}$ and $\theta_0=id$.  For the inductive step suppose that $\alpha < \lambda$ and that we have defined $g_n:X_n \to Y_n$ for $n < \alpha$ naturally in $n$, as well as invertible 2-cells $\phi_n:f_n \cong g_n$ satisfying the modification equation at $m < n < \alpha$.  In summary, we have an invertible modification $\theta_{< \alpha}:f_{< \alpha} \cong g_{< \alpha}$.  To prove the result, by transfinite induction, we must extend it to $\theta_{\leq \alpha}:f_{\leq \alpha} \cong g_{ \leq \alpha}$.

If $\alpha = \beta + 1$ we have the diagram below left.
$$\xy
(0,0)*+{X_{\beta+1}}="11";
(22,0)*+{Y_{\beta+1}}="31"; (0,-20)*+{X_{\beta}}="12";(22,-20)*+{Y_{\beta}}="32";
{\ar^{y^{\beta+1}_{\beta}} "31"; "32"}; 
{\ar@/^1pc/^{f_\beta} "12"; "32"}; 
{\ar@/_1pc/_{g_\beta} "12"; "32"}; 
{\ar_{x^{\beta+1}_{\beta}} "11"; "12"}; 
{\ar^{f_{\beta+1}} "11"; "31"}; 
{\ar@{=>}_{\theta_\beta}(12,-18)*+{};(12,-23)*+{}};
{\ar@{=>}_{f^{\beta+1}_{\beta}}(12,-4)*+{};(12,-10)*+{}};
\endxy
\hspace{1cm}
\xy
(0,-10)*+{=};
\endxy
\hspace{1cm}
\xy
(0,0)*+{X_{\beta+1}}="11";
(22,0)*+{Y_{\beta+1}}="31"; (0,-20)*+{X_{\beta}}="12";(22,-20)*+{Y_{\beta}}="32";
{\ar^{y^{\beta+1}_{\beta}} "31"; "32"}; 
{\ar_{g_\beta} "12"; "32"}; 
{\ar_{x^{\beta+1}_{\beta}} "11"; "12"}; 
{\ar@/^1pc/^{f_{\beta+1}} "11"; "31"}; 
{\ar@/_1pc/_{g_{\beta+1}} "11"; "31"}; 
{\ar@{=}^{}(10,-10)*+{};(10,-16)*+{}};
{\ar@{=>}_{\theta_{\beta+1}}(13,3)*+{};(13,-2)*+{}};
\endxy
$$

Since $\C(X_{\beta + 1},y^{\beta +1}_{\beta}):\C(X_{\beta + 1},Y_{\beta +1}) \to \C(X_{\beta + 1},Y_{\beta})$ is an isofibration there exists $\theta_{\beta +1}$ and $g_{\beta+1}$ as on the right, making the two diagrams equal.  The equation is precisely the modification condition at $\beta < \beta+1$ so that we have $g_{ \leq \beta+1}:X_{\leq \beta+1} \to Y_{\leq \beta+1}$ and $\theta_{\leq \beta+1}:f_{\leq \beta+1} \cong g_{\leq \beta+1}$, as required.

If $\alpha$ is a limit ordinal then since $Y_{\alpha} = lim_{n < \alpha} Y_{n}$ we obtain a unique morphism $g_{\alpha}:X_{\alpha} \to Y_{\alpha}$ satisfying $y^{\alpha}_{n}  \circ g_{\alpha} = g_{n} \circ y^{\alpha}_{n}$ for each $n < \alpha$.  This gives the desired cochain map $g_{\leq \alpha}:X_{\leq \alpha} \to Y_{\leq \alpha}$.  It remains to provide an invertible 2-cell $\theta_{\alpha}:f_{\alpha} \cong g_{\alpha}$ satisfying the modification condition 
\begin{equation}\label{eq:mod}
\xy
(0,0)*+{X_{\alpha}}="11";
(22,0)*+{Y_{\alpha}}="31"; (0,-20)*+{X_{n}}="12";(22,-20)*+{Y_{n}}="32";
{\ar^{y^{\alpha}_{n}} "31"; "32"}; 
{\ar_{g_n} "12"; "32"}; 
{\ar_{x^{\alpha}_{n}} "11"; "12"}; 
{\ar@/^1pc/^{f_{\alpha}} "11"; "31"}; 
{\ar@/_1pc/_{g_{\alpha}} "11"; "31"}; 
{\ar@{=}^{}(10,-10)*+{};(10,-16)*+{}};
{\ar@{=>}_{\theta_{\alpha}}(13,3)*+{};(13,-2)*+{}};
\endxy
\hspace{1cm}
\xy
(0,-10)*+{=};
\endxy
\hspace{1cm}
\xy
(0,0)*+{X_{\alpha}}="11";
(22,0)*+{Y_{\alpha}}="31"; (0,-20)*+{X_{n}}="12";(22,-20)*+{Y_{n}}="32";
{\ar^{y^{\alpha}_{n}} "31"; "32"}; 
{\ar@/^1pc/^{f_n} "12"; "32"}; 
{\ar@/_1pc/_{g_n} "12"; "32"}; 
{\ar_{x^{\alpha}_{n}} "11"; "12"}; 
{\ar^{f_{\alpha}} "11"; "31"}; 
{\ar@{=>}_{\theta_n}(12,-18)*+{};(12,-23)*+{}};
{\ar@{=>}_{f^{\alpha}_{n}}(12,-4)*+{};(12,-10)*+{}};
\endxy
\end{equation}
for each $n < \alpha$.  Now the composite 2-cells in the diagram above right are of the form $$\phi_{n}:y^{\alpha}_{n} \circ f_{\alpha} \Rightarrow g_{n} \circ x^{\alpha}_{n} = y^{\alpha}_{n} \circ g_{\alpha}$$
and indeed satisfy $y^{n}_{m} \circ \phi_{n} = \phi_{m}$ for all $m < n$.  This is proven by the following equations of 2-cells
$$
\xy
(0,0)*+{X_{\alpha}}="11";
(22,0)*+{Y_{\alpha}}="31"; (0,-20)*+{X_{n}}="12";(22,-20)*+{Y_{n}}="32";
(0,-40)*+{X_{m}}="13";(22,-40)*+{Y_{m}}="33";
{\ar^{y^{\alpha}_{n}} "31"; "32"}; 
{\ar@/^1pc/^{f_n} "12"; "32"}; 
{\ar@/_1pc/_{g_n} "12"; "32"}; 
{\ar_{x^{\alpha}_{n}} "11"; "12"}; 
{\ar^{f_{\alpha}} "11"; "31"}; 
{\ar@{=>}_{\theta_n}(12,-18)*+{};(12,-23)*+{}};
{\ar@{=>}_{f^{\alpha}_{n}}(12,-4)*+{};(12,-10)*+{}};
{\ar_{x^{n}_{m}} "12"; "13"}; 
{\ar^{y^{\alpha}_{n}} "32"; "33"}; 
{\ar_{g_m} "13"; "33"};
{\ar@{=}^{}(10,-30)*+{};(10,-36)*+{}};
\endxy
\hspace{0.5cm}
\xy
(0,-20)*+{=};
\endxy
\hspace{0.5cm}
\xy
(0,0)*+{X_{\alpha}}="11";
(22,0)*+{Y_{\alpha}}="31"; (0,-20)*+{X_{n}}="12";(22,-20)*+{Y_{n}}="32";
(0,-40)*+{X_{m}}="13";(22,-40)*+{Y_{m}}="33";
{\ar^{y^{\alpha}_{n}} "31"; "32"}; 
{\ar^{f_n} "12"; "32"}; 
{\ar_{x^{\alpha}_{n}} "11"; "12"}; 
{\ar^{f_{\alpha}} "11"; "31"}; 
{\ar@{=>}_{f^{\alpha}_{n}}(12,-4)*+{};(12,-10)*+{}};
{\ar_{x^{n}_{m}} "12"; "13"}; 
{\ar^{y^{\alpha}_{n}} "32"; "33"}; 
{\ar@/^1pc/^{f_m} "13"; "33"}; 
{\ar@/_1pc/_{g_m} "13"; "33"};
{\ar@{=>}_{\theta_m}(12,-38)*+{};(12,-43)*+{}};
{\ar@{=>}_{f^{n}_{m}}(12,-24)*+{};(12,-30)*+{}};
\endxy
\hspace{0.5cm}
\xy
(0,-20)*+{=};
\endxy
\xy
(0,0)*+{X_{\alpha}}="11";
(22,0)*+{Y_{\alpha}}="31"; (0,-20)*+{X_{m}}="12";(22,-20)*+{Y_{m}}="32";
{\ar^{y^{\alpha}_{m}} "31"; "32"}; 
{\ar@/^1pc/^{f_n} "12"; "32"}; 
{\ar@/_1pc/_{g_n} "12"; "32"}; 
{\ar_{x^{\alpha}_{m}} "11"; "12"}; 
{\ar^{f_{\alpha}} "11"; "31"}; 
{\ar@{=>}_{\theta_m}(12,-18)*+{};(12,-23)*+{}};
{\ar@{=>}_{f^{\alpha}_{m}}(12,-4)*+{};(12,-10)*+{}};
\endxy
$$
of which the first uses that $\theta:f_{< \alpha} \Rightarrow g_{< \alpha}$ is a modification and the second uses pseudonaturality of $f$.  By the universal property of the limit $\C(X_{\alpha},Y_{\alpha})$ in $\CAT$ therefore there exists a unique 2-cell $\theta_{\alpha}$ which, on postcomposition by the $y^{\alpha}_{n}$ equals $\phi_{n}$; in other words, satisfying the required equation ~\eqref{eq:mod} for a modification depicted above.
\end{proof}

\begin{Proposition}\label{prop:bilimit}
\begin{enumerate}
\item
If $X:\lambda^{op} \to \C$ is a smooth isofibrant cochain then at each limit ordinal $\alpha < \lambda$ the limit $X_{\alpha}=lim_{n < \alpha} X_n$ is a bilimit.
\item Each pseudonatural equivalence between smooth isofibrant $\lambda$-cochains is isomorphic (via a modification) to a cochain equivalence.   In particular, cochain equivalence is an equivalence relation, when restricted to smooth isofibrant cochains.
\end{enumerate}
\end{Proposition}
\begin{proof}
\begin{enumerate}
\item To say that the cone $\Delta(X_{\alpha}) \to X_{< \alpha} \in [\alpha^{op},\C]$ exhibits the limit $X_{\alpha}$ as a \emph{bilimit} is, by definition, to say that the composite 
$$\C(A,X_{\alpha}) \to [\alpha^{op},\C](\Delta A,X_{< \alpha}) \hookrightarrow Ps(\alpha^{op},\C)(\Delta A,X_{< \alpha})$$
is an equivalence for each $A$.  Since the first component is an isomorphism, this will be the case just when the second component is an equivalence.  Since it is always fully faithful, this is will be the case just when it is essentially surjective on objects and, by the preceding lemma, this holds since $X_{< \alpha}$ is smooth isofibrant.
\item If $f:X \rightsquigarrow Y$ is a pseudonatural equivalence, then by the lemma there is a 2-natural transformation $g:X \to Y$ and invertible modification $f \cong g$.  For $n \in \lambda$ we have $f_n \cong g_n$ since each $f_n$ is an equivalence so too is the naturally isomorphic $g_n$.  This proves the first part.  The second part from the first part using that pseudonatural equivalence is an equivalence relation.
\end{enumerate}
\end{proof}

\end{document}